\documentclass[11pt]{amsart}
\usepackage[T1]{fontenc}
\usepackage{amsmath,amssymb,amsfonts}
\usepackage[all,cmtip]{xy}
\usepackage[english]{babel}

\usepackage[normalem]{ulem}

\usepackage{hyperref}

\numberwithin{equation}{section}

\newcommand{\norm}{\|\,.\,\|}

\newcommand{\Ka}{K^{\rm an}}
\newcommand{\Kto}{K^{\rm top}}
\newcommand{\ka}{k^{\rm an}}
\newcommand{\kto}{k^{\rm top}}
\newcommand{\Kco}{K^{\rm cont}}
\newcommand{\kco}{k^{\rm cont}}
\newcommand{\Sch}{{\rm \bf Sch}}

\DeclareMathOperator{\GL}{GL}
\DeclareMathOperator{\Mat}{M}

\DeclareMathOperator{\Spec}{Spec}
\DeclareMathOperator*{\colim}{colim}

\DeclareMathOperator{\coker}{coker}

\DeclareMathOperator{\im}{im}
\DeclareMathOperator{\diag}{diag}

\DeclareMathOperator{\cofib}{cofib}
\DeclareMathOperator{\fib}{fib}
\DeclareMathOperator{\mycup}{\smallsmile}

\newcommand{\bC}{\mathbf{C}}
\DeclareMathOperator{\Pro}{Pro}
\DeclareMathOperator{\Fun}{Fun}
\newcommand{\Spaces}{\mathbf{Spc}}
\DeclareMathOperator{\Nerve}{N}
\DeclareMathOperator{\Map}{Map}
\newcommand{\lex}{\mathrm{lex}}
\newcommand{\acc}{\mathrm{acc}}

\newcommand{\an}{\mathrm{an}}
\newcommand{\cont}{\mathrm{cont}}
\newcommand{\op}{\mathrm{op}}
\newcommand{\sSet}{\mathbf{sSet}}

\newcommand{\Spc}{\operatorname{Sp}}

\newcommand{\prolim}[1]{\underset{#1}{\operatorname{``}\lim \operatorname{''}  \hspace{.2ex}  }}
\newcommand{\lprolim}[1]{\operatorname{``}\lim \operatorname{''} \hspace{-1.9ex} {}_{#1} \hspace{.5ex} }

\newcommand{\Sp}{\mathbf{Sp}}

\newcommand{\KGL}{K^{\GL}}
\newcommand{\BGL}{\mathrm{BGL}}
\newcommand{\N}{{\mathbb N}}

\newcommand{\Z}{{\mathbb Z}}

\newcommand{\cO}{\mathcal{O}}
\newcommand{\R}{{\mathbb R}}
\newcommand{\sU}{{\mathcal U}}

\newcommand{\Ab}{\mathbf{Ab}}

\def\At{A\langle t \rangle}
\def\Ati{A\langle t^{-1} \rangle}
\def\Atti{A\langle t, t^{-1} \rangle}
\def\Ao{A_0}

\newcommand{\bark}{{\bar k}}
\newcommand{\tildek}{{\tilde k}}
\newcommand{\hNpi}{\hat  N^\pi}

\theoremstyle{plain}
\newtheorem{thm}{Theorem}[section]

\newtheorem{lem}[thm]{Lemma}
\newtheorem{lemma}[thm]{Lemma}
\newtheorem{cor}[thm]{Corollary}
\newtheorem{prop}[thm]{Proposition}

\theoremstyle{definition}
\newtheorem{ex}[thm]{Example}
\newtheorem{defn}[thm]{Definition}

\newtheorem{claim}[thm]{Claim}
\newtheorem{rmk}[thm]{Remark}
\newtheorem{rmks}[thm]{Remarks}

\title{\textit{K}-theory of non-archimedean rings I}
\author{Moritz Kerz}
\address{Fakult\"at f\"ur Mathematik, Universit\"at Regensburg, 93040 Regensburg, Germany}
\email{moritz.kerz@ur.de}

\author{Shuji Saito}
\address{Interactive Research Center of Science,
Graduate School of Science and Engineering,
Tokyo Institute of Technology,
2-12-1 Okayama, Meguro,
Tokyo 152-8551,
Japan}
\email{sshuji@msb.biglobe.ne.jp}

\author{Georg Tamme}
\address{Fakult\"at f\"ur Mathematik, Universit\"at Regensburg, 93040 Regensburg, Germany}
\email{georg.tamme@ur.de}
\thanks{The authors are supported by the DFG through CRC 1085 \textit{Higher Invariants} (Universit\"at Regensburg). 
The second author thanks the Department of Mathematics at the University of
Regensburg
for hospitality while this work was done.
He is supported by JSPS KAKENHI Grant (15H03606).}

\date{\today}

\begin{document}

\begin{abstract}
We introduce a variant of homotopy $K$-theory for Tate rings, which we call \emph{analytic $K$-theory}. It is homotopy invariant with respect to the analytic affine line viewed as an ind-object of closed disks of increasing radii. Under a certain regularity assumption we prove an analytic analog of the Bass fundamental theorem and we compare analytic $K$-theory with continuous $K$-theory, which is defined in terms of models. Along the way we also prove some results about the algebraic $K$-theory of Tate rings.
\end{abstract}

\maketitle

\setcounter{tocdepth}{2}

\section{Introduction}

In this note we introduce and study a new version of $K$-theory for non-archimedean rings,
which takes the topology into account. As a motivation, we first recall some 
facts about continuous $K$-theory and about classical topological $K$-theory before we describe our construction in more detail in Subsection~\ref{intro:anaK}.

\subsection{Motivation: Algebraization problem and topological \textit{K}-theory}\label{into.sub1}

Let $\kappa$ be a complete discretely valued field, $\kappa_0\subset \kappa$ be its ring
of integers and $\pi\in \kappa_0$ be a prime element.

Let $X$ be a proper smooth scheme over $\kappa_0$.
The algebraization problem for $K_0$-classes of vector bundles asks how we can describe
the image of the map
\begin{equation} \label{eqin:alg}
  K_0(X) \to \lim_n K_0(X_n),
\end{equation}
where $X_n=X\otimes_{\kappa_0} \kappa_0/(\pi^n)$.

Unfortunately, the map \eqref{eqin:alg} is quite far
from being surjective in general, see~\cite[App.~B]{BEK}. However, one can ask if for an
abelian scheme $X$ over $\kappa_0$ and an integer $m>1$ the image of the map~\eqref{eqin:alg}, after tensoring with $\mathbb{Q}$, contains the
kernel of $\psi^{m^2}-[m]^*$, where $\psi$ is the Adams operation and $[m]:X\to X$ is
multiplication by $m$.

In~\cite{BEK} the first author together with Bloch and Esnault showed that if the answer
to this question is positive for all abelian schemes in
characteristic zero then the Hodge conjecture for abelian varieties holds.

These observations motivate us to take a closer look at the codomain of~\eqref{eqin:alg},
which is a kind of continuous $K$-group.
Such continuous $K$-groups have been extensively studied by several authors with different methods,  e.g.~\cite{Panin, HM1, Geisser-Hesselholt, GH2, Beilinson-continuous}.
 In this article we study such continuous
$K$-theory of affine formal schemes over $\kappa_0$ and of affinoid spaces over $\kappa$
and relate it to what we call analytic $K$-theory. The construction of the latter is in turn motivated by the
following observation about topological $K$-theory.

\medskip

Let $X$ be a compact Hausdorff space. The (complex) topological $K$-theory  of $X$ is usually defined
as the mapping spectrum
\[
\Kto(X) = \mathrm{map}( \Sigma^\infty X_+ , KU),
\]
where $KU$ is the spectrum representing complex $K$-theory. The connective covering spectrum
of $\Kto(X)$ is denoted by $ \kto (X)$.

For a ring $A$ let $K(A)$ be the non-connective algebraic $K$-theory
spectrum of $A$ and $k(A)$ be its connective cover.
 Let $C(X)$ be the ring of complex continuous functions on the compact Hausdorff
space $X$ and let
$j\mapsto \Delta^j_{\rm top}$ be the cosimplicial space of standard simplices.

An important observation is that topological $K$-theory can be obtained from algebraic
$K$-theory by making it homotopy invariant. Recall that there is a canonical morphism of
spectra $k(C(X))\to \kto (X)$.
As $\kto$ is homotopy invariant, it is natural to study the induced morphism
\[\psi\colon k(C(X\times \Delta_{\rm top})) \to \kto (X),\] where the domain is the realization
of the simplicial spectrum given by $[m]\mapsto k(C(X\times \Delta_{\rm top}^m))$.
Our definition of non-archimedean analytic $K$-theory below is motivated by the fact that
$\psi$ is an equivalence of spectra.

\subsection{Analytic\textit{K}-theory and Karoubi--Villamayor \textit{K}-theory}
	\label{intro:anaK}

There exist several articles generalizing aspects of either classical homotopy theory of the real
interval ${[0,1]}\subset \mathbb R$ or of $\mathbb A^1$-homotopy theory from algebraic geometry to the non-archimedean
setting, for example~\cite{Karoubi-Villamayor},  \cite{Karoubi1971}, \cite{calvo}, \cite{ayoub}. In all those works the approach has been to use the rigid analytic closed unit
ball $\mathbb B_\kappa$ over the  non-archimedean complete  valued field $\kappa$ (or some
variant of it in the work of Karoubi and Villamayor) as  an interval for homotopy theory.

The novel aspect we suggest in this note is to use the affinoid exhaustion of the whole
analytic line
\begin{equation}\label{eq.intro.exhaust}
  (\mathbb A^1)_\kappa^\an=\colim_\rho\, \mathbb B_\kappa (\rho)
\end{equation}
instead of the rigid analytic closed unit ball
$\mathbb B_\kappa= \mathbb B_\kappa (1)$ to construct a homotopy theory, or more precisely
a non-archimedean $K$-theory. Here $\mathbb B_\kappa (\rho)$ is the
rigid analytic closed ball of radius $\rho$ centered at the origin. So in some sense this new homotopy
theory keeps more information than the previous approaches mentioned above. Our approach is
motivated by the observation that the Picard group of a smooth affinoid algebra is
homotopy invariant in our sense, while it is not  $\mathbb B_\kappa$-invariant in general, see~\cite{proHIPic}.

More concretely, we mimic the description of topological $K$-theory explained in
Subsection~\ref{into.sub1} in order to construct  a new analytic $K$-theory of
affinoid spaces. Instead of actually performing the limit in~\eqref{eq.intro.exhaust} we
work with pro-systems of algebraic $K$-groups.

Let $\kappa$ be a complete discretely valued field and let $X$ be an affinoid space over
$\kappa$. For $1\le \rho\in |\kappa^\times|$ let $\Delta_\rho^m$ be the standard rigid $m$-simplex
of radius $\rho$, i.e.
\[
\Delta_\rho^m  = \Spc \kappa\langle X_0,\ldots , X_m \rangle_\rho/ (X_0 + \cdots + X_m-1)
\]
where $\kappa\langle X_0,\ldots , X_m \rangle_\rho$ is the Tate algebra of formal power series
which converge on a closed polydisc of radius $\rho$.

By varying $\rho\ge 1$ we obtain an ind-object  $\rho\mapsto \Delta_\rho$ in the category of cosimplicial
affinoid spaces over $\kappa$. We suggest to define connective analytic $K$-theory of the
affinoid space $X$ as the pro-spectrum
\[
\ka (X) = \prolim{\rho } k(\cO(X\times \Delta_\rho)).
\]
Here $\cO(Y)$ denotes the ring of rigid analytic functions of an affinoid space $Y$.
There exists also a Karoubi--Villamayor analog of $\ka(X)$ defined as the
pro-simplicial set
\[
KV^\an (X) = \prolim{\rho} \BGL(\cO(X\times \Delta_\rho)).
\]
For  $X=\Spc A$   we also write $\ka (A)$ for $\ka(X)$  and similarly  $KV^\an(A)$ for $KV^\an(X)$. We
write $\ka_i(A)$ resp.\ $KV^\an_i(A)$ for the homotopy pro-group $\pi_i$ of $\ka(A)$ resp.~of $KV^\an(A)$.

This definition of non-archimedean analytic $K$-theory has very pleasant properties if $A$ is
regular. For technical reasons we also have to assume a certain weak form of resolution of
singularities for a ring of definition of $A$, which we call condition $(\dagger)_A$, see Subsection~\ref{subsec:nonarch.tate} for details. This condition
is satisfied for any regular affinoid algebra $A$ if $\kappa$ has residue characteristic zero by a
variant of Hironaka's resolution of singularities.

\begin{thm}\label{intro.mainthm1}
  For a regular affinoid algebra $A$ which satisfies condition $(\dagger)_A$ we obtain:
\begin{itemize}
\item[(i)]  $\ka_0(A)$ is isomorphic to the constant pro-group $K_0(A)$.
\item[(ii)] There are isomorphisms of pro-groups
  \[
    \ka_1(A)  \cong \prolim{0<r<1} K_1(A)/(1+A(r )) = \prolim{\rho>1}\GL(A)/\GL(A)_\rho,
  \]
    where $ \GL(A)_\rho$ is the normal subgroup generated by matrices
    $g$ satisfying the unipotence condition $\lim_n \Vert (g-1)^n \Vert \, \rho^n=0$ and where $A(r)=\{
    a\in A \, |\, \Vert a \Vert < r \}$ for $r>0$.
\item[(iii)] For $i>0$ there is an isomorphism $ KV^\an_i(A) \xrightarrow{\simeq} \ka_i(A) $
  of pro-groups.
\item[(iv)] $\ka$ satisfies the analytic analog of the Bass fundamental theorem of algebraic
  $K$-theory, i.e.\ for $i>0$ there is a canonical exact sequence
  \[
    0\to \ka_i(A)\to \ka_i(A\langle t \rangle ) \times \ka_i(A\langle t^{-1} \rangle )\to
    \ka_i(A\langle t,t^{-1} \rangle ) \to \ka_{i-1}(A) \to 0
  \]
  together with a canonical splitting of the surjection on the right. 
 
\end{itemize}
\end{thm}

For an  affinoid $\kappa$-algebra $A$ the choice of a surjection $\kappa \langle t_{1},\dots t_{n} \rangle_{1} \twoheadrightarrow A$ gives rise to a residue norm $\| . \|$ on $A$. Up to equivalence, it does not depend on the chosen surjection. This is the norm occurring in (ii).
The results of Theorem~\ref{intro.mainthm1} are obtained at different places throughout the paper. We collect the references in Section~\ref{sec:overview} below.

\begin{rmk}
  Note that (i) and (ii) of Theorem~\ref{intro.mainthm1} are analogous to classical
  properties of topological $K$-theory, see~\cite[Sec.~3.2]{rosenberg}. We will see in a
  sequel to this note that (iv) can be interpreted as the Mayer--Vietoris sequence of
  $\ka$ corresponding to the standard admissible covering of the rigid space
  $(\mathbb P^1_A)^\an$.
\end{rmk}

Theorem~\ref{intro.mainthm1}(iv) implies that we can use a    delooping
  process (see Remark~\ref{rmk.deloop} for details) to define  a non-connective analytic $K$-theory spectrum $\Ka (A)$
such that under the same conditions as in Theorem~\ref{intro.mainthm1} the pro-spectrum $\ka (A) $ is the connective covering of $ \Ka (A)$ and such that the analytic
Bass fundamental theorem holds, i.e.\ the analog of the exact sequence in
Theorem~\ref{intro.mainthm1}(iv) for $\Ka_i$ exists for all $i\in \mathbb Z$.

\subsection{Continuous \textit{K}-theory}
	\label{intro:contK}

Let $A$ be an affinoid algebra over the complete discretely valued field $\kappa$.
Let $\kappa_0\subset \kappa$ be the ring of integers and $\pi\in \kappa_0$ be a prime element. Choose a
$\kappa_0$-algebra $A_0$ flat and topologically of finite type over $\kappa_0$ with $A_0[1/\pi]
\cong A$, see \cite{bosch}. Such an $A_0$ is also called a ring of definition of $A$. Morrow
\cite{morrow-hist} defines continuous $K$-theory of $A$ as
\[
\Kco (A) = \cofib (K(A_0 \text{ on } (\pi)) \to \prolim{n} K(A_0/(\pi^n) ),
\]
where $K$ stands for the non-connective algebraic $K$-theory spectrum and where the
cofibre is taken in the $\infty$-category of pro-spectra, recalled in Section~\ref{sec:prelims}. Using
pro-excision one shows that the pro-group $\Kco_i (A)$ defined as $\pi_i( \Kco (A))$
 for $i\in \mathbb Z$ does not depend on the choice of the ring of definition $A_0$ up to
 canonical isomorphism.

In fact, variants of the pro-group $\Kco_i (A)$ have been studied since the introduction of
algebraic $K$-theory, for example we have:

\begin{itemize}
\item[(i)] The canonical map $K_0(A) \xrightarrow{\simeq} \Kco_0(A)$ is an isomorphism of pro-groups, see
  Proposition~\ref{prop.comalgcont}. For $i\le 0$ the pro-groups $\Kco_i(A)$ are constant and
  agree with the negative ``topological'' $K$-groups studied
  in~\cite[Sec.~7]{Karoubi-Villamayor} and in~\cite{calvo}.
\item[(ii)] The canonical map $  \lprolim{n} K_1(A)/(1+\pi^n A_0)  \xrightarrow{\simeq}
    \Kco_1 (A)$ is an isomorphism of pro-groups, see Lemma~\ref{lem.calK1cont}. In
    particular, for a complete discretely valued field $\kappa$ with ring of integers
    $\kappa_0$ we  get $ \Kco_1 (\kappa) \cong  \lprolim{n} \kappa^\times /(1+\pi^n
    \kappa_0) $.
\item[(iii)] For a complete discretely valued field $\kappa$ with ring of integers
    $\kappa_0$ the canonical map $\lprolim{n} K_2(\kappa)/\{ 1+\pi^n \kappa_0 , \kappa_0^\times \}
    \xrightarrow{\simeq}   \Kco_2 (\kappa)$ is an isomorphism of pro-groups, see Lemma~\ref{lem.k2cont}. In particular, if
    the residue field of $\kappa_0$ is finite, the norm-residue symbol induces an
    isomorphism $\Kco_2 (\kappa)\cong \mu\subset \kappa^\times$, where $\mu$ are the roots
    of unity (Moore's Theorem), see~\cite[App.]{milnor}.
    
\end{itemize}

The key observation of our note is that for regular affinoid algebras continuous
$K$-theory agrees with analytic $K$-theory. This  relates the two existing
approaches to the
$K$-theory of non-archimedean rings which originate from~\cite{Karoubi-Villamayor}.

\begin{thm}
	\label{intro.thm.KcontKan}
For a regular affinoid algebra $A$ which satisfies condition $(\dagger)_A$ the pro-spectra
  $\Ka (A)$ and  $\Kco(A)$ are weakly equivalent.
\end{thm}

For the definition of weakly equivalent  pro-spectra see
Subsection~\ref{sec:weak-equivalences-of-pro-spectra}. In particular, this means that the
pro-groups $ \Ka_i (A) $ and $ \Kco_i (A)$ are naturally isomorphic.
This theorem is proven as Theorem~\ref{thm:Kcont-Kan} below.

In a sequel to this note we will discuss the globalization of our construction to a
non-affine formal scheme $X$ over $\kappa_0$   and to its associated rigid space
$X_\kappa$ in the sense of Raynaud, see~\cite[Ch.~8]{bosch}.
The relevance of rigid geometry in the study of continuous $K$-groups is suggested by the
following consequence of pro-cdh descent for algebraic $K$-theory, see~\cite[Thm.~A]{KST-Weibel}: the homotopy pro-groups of the continuous $K$-theory pro-spectrum
\[
\Kco (X_\kappa) = \cofib (K(X \text{ on } (\pi)) \to \prolim{n} K(X/(\pi^n) ),
\]
depend only on $X_\kappa$ and not on the choice of $X$. This  
is also sketched in~\cite{morrow-hist}.

\subsection{Overview of sections}
	\label{sec:overview}

In Section~\ref{sec:prelims} we collect various facts about pro-objects in $\infty$-categories, and in particular about pro-spectra and pro-spaces.
Most of our results are formulated for adic rings or Tate rings --- we use affinoid
algebras in the introduction only for simplicity of presentation.   Basic properties of these non-archimedean rings are discussed 
in Section~\ref{sec:non-archimedean-rings}. There we also introduce normed rings which are used in the study of the Karoubi-Villamayor version of our theory in Section~\ref{sec.kv}.
In Section~\ref{sec:K-theory-of-na-rings} we discuss some general preliminaries about
algebraic $K$-theory and we study the low degree algebraic $K$-groups of Tate rings.
Continuous $K$-theory is introduced in Section~\ref{sec.contK}. Beside the basic properties mentioned in \ref{intro:contK} above, we also prove an analytic version of the Bass fundamental theorem and study homotopy invariance of continuous $K$-theory.
In Section~\ref{sec.anK} we introduce the  analytic $K$-theory discussed in~\ref{intro:anaK}. The main result there is the comparison with continuous $K$-theory for a regular affinoid algebra~$A$ satisfying $(\dag)_{A}$.

\begin{proof}[\normalfont{We end this overview with the} \emph{proof of Theorem~\ref{intro.mainthm1}}]
Assertion (i) is Lem\-ma~\ref{lem:K0-K0an}(ii). By Theorem~\ref{intro.thm.KcontKan} we can identify analytic $K$-theory with continuous $K$-theory. Using this identification, Lemma~\ref{lem.calK1cont} gives the first isomorphism of part (ii) and Proposition~\ref{prop.comalgcont}(i) gives (iv). Assertion (iii) is precisely Lemma~\ref{lem:KV-Kan-connective-cover}. Now Lemma~\ref{lem:KV1-explicit} gives the second isomorphism in (ii).
\end{proof}

\subsection*{Acknowledgment}

We would like to thank  M.~Morrow for helpful discussions on $K$-theory of non-archimedean
rings. In particular, he suggested the results discussed in
Subsection~\ref{subsec.alghomo} and he proposed the definition of continuous $K$-theory of
affinoid algebras.
We would like to thank the referee for many comments helping to improve the paper.


\section{Topological preliminaries}
\label{sec:prelims}

In the study of $K$-theory of non-archimedean rings, pro-spectra play an important role.
The theory of $\infty$-categories provides a well developed and flexible framework to deal with homotopy theoretic questions about pro-spectra. In the following, we therefore collect some  facts about pro-objects in $\infty$-categories in general and about pro-spectra.
General references for pro-objects are \cite{AM} or \cite{Isaksen-Calculating}.
The $\infty$-categorical version is discussed for example in \cite[\S A.8.1]{sag}.
We write $\Spaces$ for the $\infty$-category of spaces.

\subsection{Pro-objects in \texorpdfstring{$\infty$-}{infinity }categories}

Let $\bC$ be an $\infty$-category. 
For technical reasons we assume that $\bC$ is accessible \cite[Def.~5.4.2.1]{HTT} and admits finite limits.
A pro-object in $\bC$ is represented by a diagram $X \colon I \to \bC$ where $I$ is a small cofiltered $\infty$-category. We also use the notation $X = \lprolim{I} X_{i}$ and call this a level-representation of $X$.
The pro-objects are  the objects of an $\infty$-category $\Pro(\bC)$.
If $Y = \lprolim{J} Y_{j}$ is a second pro-object, the mapping space is given by the classical formula
\begin{equation}\label{eq:mapping-spaces-in-Pro-cat}
\Map(\prolim{I} X_{i}, \prolim{J} Y_{j}) \simeq \lim_{J} \colim_{I} \Map(X_{i},Y_{j}),
\end{equation}
where the limit and colimit are computed in the $\infty$-category of spaces.
A precise definition is
\begin{equation}
	\label{eq.def.pro}
\Pro(\bC) = \Fun^{\lex,\acc}(\bC,\Spaces)^{\op}	
\end{equation}
where $\Fun^{\lex,\acc}(\bC,\Spaces)$ is  the full subcategory of $\Fun(\bC,\Spaces)$ spanned by those functors which are accessible (i.e.~preserve small $\kappa$-filtered colimits for some regular cardinal $\kappa$) and left exact (i.e.~commute with finite limits). 
Here the pro-object $X$ from above corresponds to the functor $\colim_{I} \Map(X_{i}, -) \colon \bC \to \Spaces$, see
 \cite[Def.~A.8.1.1, Rem.~A.8.1.5]{sag}. 

Any functor $f \colon \bC \to \mathbf{D}$ induces a functor $\Pro(f)\colon \Pro(\bC) \to \Pro(\mathbf{D})$ given by $\lprolim{I} X_{i} \mapsto \lprolim{I} f(X_{i})$, see \cite[Ex.~A.8.1.8]{sag}.

Finite limits and colimits in $\Pro(\bC)$ can be computed level-wise:
Let $K$ be a finite simplicial set. The functor $\lim_{K} \colon \Fun(K,\bC) \to \bC$ induces a functor $\Pro(\Fun(K,\bC)) \to \Pro(\bC)$, the `level-wise limit'. If $\bC$ admits finite colimits, we similarly have a level-wise  colimit.
Any pro-object of $K$-diagrams induces a $K$-diagram of pro-objects via the natural functor
$\Pro(\Fun(K,\bC)) \to \Fun(K, \Pro(\bC))$.
\begin{lemma} \label{lem:level-wise-limit}
Let $F  \in \Pro(\Fun(K,\bC))$.
Then the level-wise limit of $F$ is a limit of the induced diagram $K \to \Pro(\bC)$. If $\bC$ admits finite colimits, the same is true for the colimit.
\end{lemma}
\begin{proof}
This can be checked as in the 1-categorical case (see \cite[App., Prop.~4.1]{AM}) using  formula \eqref{eq:mapping-spaces-in-Pro-cat} and the fact that in $\Spaces$ finite limits commute with filtered colimits \cite[Prop.~5.3.3.3]{HTT}.
\end{proof}

We also have level-representations for certain diagrams. In the following Lemma $\Nerve(A)$ denotes the nerve of a partially ordered set $A$ viewed as a category.
\begin{lemma}\label{lem:finite-diagrams}
Let $A$ be a finite partially ordered set. The natural functor
\[
\Pro( \Fun(\Nerve(A), \bC) ) \to \Fun( \Nerve(A), \Pro(\bC) )
\]
is essentially surjective.
\end{lemma}
\begin{proof}
If $\bC$ is small, this follows directly from \cite[Prop.~5.3.5.15]{HTT}. The general case can be reduced to this as in the proof of \cite[Prop.~A.8.1.6]{sag}.
\end{proof}

\begin{lemma}\label{lem:complete-and-cocomplete}
The $\infty$-category $\Pro(\bC)$ is complete. If $\bC$ admits finite resp.~all small colimits, then so does $\Pro(\bC)$. 
\end{lemma}
\begin{proof}
Lemma~\ref{lem:level-wise-limit} implies in particular that $\bC \to \Pro(\bC)$ commutes with finite limits and colimits. In particular, $\Pro(\bC)$ has a final  object.
Combining Lemmas~\ref{lem:finite-diagrams} and \ref{lem:level-wise-limit} we see that $\Pro(\bC)$ admits pullbacks. Now \cite[Cor.~4.4.2.4]{HTT} implies that $\Pro(\bC)$ has all finite limits. 
If $\bC$ admits finite colimits, then by the same argument $\Pro(\bC)$ has all finite colimits.
Since $\Pro(\bC)$ also has all small cofiltered limits, it follows that $\Pro(\bC)$ is complete by \cite[Prop.~4.4.2.6]{HTT}. In order to show that $\Pro(\bC)$ is cocomplete if $\bC$ is, it now suffices to show that $\Pro(\bC)$ has small coproducts \cite[Prop.~4.4.2.6]{HTT}. This is shown in Lemma~\ref{lem:coproducts} below.
\end{proof}

Recall that an accessible $\infty$-category \cite[Def.~5.4.2.1]{HTT} which is cocomplete is called presentable \cite[Def.~5.5.0.1]{HTT}. A presentable $\infty$-category is automatically complete \cite[Cor.~5.5.2.4]{HTT}.

\begin{lemma}\label{lem:coproducts}
Assume that $\bC$ is presentable.
Let $A$ be a set, and let $(X^{\alpha})_{\alpha\in A}$ be a family of objects in $\Pro(\bC)$. Choose level representations $$X^{\alpha} = \prolim{I_{\alpha}} X^{\alpha}_{i}.$$ Define the cofiltered $\infty$-category $K = \prod_{\alpha\in A} I_{\alpha}$. Then $$X = \prolim{k\in K} \coprod_{\alpha\in A} X^{\alpha}_{k_{\alpha}}$$ is a coproduct of $(X^{\alpha})_{\alpha\in A}$.
\end{lemma}

\begin{proof} (See \cite[\S8]{Isaksen-Calculating} for the classical situation.)
  Let $Y = \lprolim{J} Y_{j}$ be any pro-object. Then
\begin{align*}
\Map(X,Y) 
&\simeq \lim_{J} \colim_{K} \prod_{\alpha\in A} \Map(X^{\alpha}_{k_{\alpha}},Y_{j} ) \\
&\simeq \lim_{J} \prod_{\alpha\in A} \colim_{I_{\alpha}} \Map( X^{\alpha}_{k_{\alpha}}, Y_{j}) &&\text{(see below)} \\
&\simeq \prod_{\alpha\in A} \lim_{J} \colim_{I_{\alpha}} \Map( X^{\alpha}_{k_{\alpha}}, Y_{j}) \\
&\simeq \prod_{\alpha\in A} \Map(X^{\alpha}, Y).
\end{align*}
The second equivalence holds since $\colim_{K}\prod_{\alpha\in A} \simeq \prod_{\alpha\in A} \colim_{I_{\alpha}}$  in the $\infty$-category of spaces. Indeed, since products and filtered colimits in spaces commute with homotopy groups, this follows from the corresponding fact in the category of sets, where it can be checked directly. 
\end{proof}

\begin{lemma}
If $\bC$ is stable, then so is $\Pro(\bC)$.
\end{lemma}
For the notion of a stable $\infty$-category see \cite[Def.~1.1.1.9]{HA}.
\begin{proof}
By Lemma~\ref{lem:level-wise-limit} $\Pro(\bC)$ is pointed. By Lemma~\ref{lem:complete-and-cocomplete} it has all finite limits and colimits. Since by Lemma~\ref{lem:level-wise-limit} again the suspension functor in $\Pro(\bC)$ can be computed level-wise, it is an equivalence. We conclude using \cite[Cor.~1.4.2.27]{HA}.
\end{proof}

\subsection{Weak equivalences of pro-spectra}
\label{sec:weak-equivalences-of-pro-spectra}

We let $\Sp$ be the presentable stable $\infty$-category of spectra \cite[\S 1.4.3]{HA}. 
It carries a t-structure whose heart is equivalent to the category $\Ab$ of abelian groups.
 In particular, an abelian group $A$ gives a spectrum $H(A)$ with $\pi_{0}(H(A)) \cong A$, $\pi_{i}(H(A)) = 0$ for $i \not=0$. 

We denote by $\Sp^{+} \subseteq \Sp$ the full stable subcategory of spectra which are bounded
above, and denote the inclusion by $\iota$. 
By \cite[Ex.~A.8.1.8]{sag} we
get an adjunction
\[
\iota^{*} : \Pro(\Sp) \leftrightarrows \Pro(\Sp^{+}) : \Pro(\iota)
\]
where $\iota^*$ is induced using \eqref{eq.def.pro} by composition with $\iota$, and 
$\Pro(\iota)$ is fully faithful.
In other words, $\Pro(\Sp^{+})$ is a localization of $\Pro(\Sp)$ in the sense of \cite[Def.~5.2.7.2]{HTT}.
In particular, $\Pro(\Sp^{+})$ also admits all small colimits.

In terms of level representations $\iota^{*}$ has the following description: Denote by $\tau_{\leq n}\colon \Sp \to \Sp_{\leq n}$ the truncation functor. It is left adjoint to the inclusion.
\begin{lemma}\label{lem:explicit-description-of-i-star}
Let $\lprolim{I} X_{i} \in \Pro(\Sp)$. Then we have a natural equivalence 
\[
\iota^{*}(\prolim{I} X_{i}) \simeq \prolim{\N\times I} \tau_{\leq n}X_{i}
\]
in $\Pro(\Sp^{+})$.
\end{lemma}
\begin{proof}
In terms of left exact accessible functors $\iota^{*}(\lprolim{I} X_{i})$ is given by 
$\colim_{I} \Map(X_{i},-) \in \Fun^{\lex,\acc}(\Sp^{+},\Spaces)$ and this colimit can be computed object-wise 
(see 
\cite[Cor.~5.1.2.3]{HTT}).
Consider $Y \in \Sp^{+}$. Note that if  $Y\in \Sp_{\leq m}$  then $\Map(\tau_{\leq n} X_{i}, Y) \simeq \Map(X_{i},Y)$ for all $n\geq m$. Hence
\[
\colim_{I} \Map(X_{i},Y) \simeq \colim_{\N\times I} \Map(\tau_{\leq n}X_{i}, Y),
\]
which implies the claim.
\end{proof}

\begin{defn}
	\label{def:weak-equivalence}
We call a map of pro-spectra $X \to Y$  a \emph{weak equivalence} if $\iota^{*} X \to \iota^* Y$
is an equivalence in $\Pro(\Sp^{+})$. We say that the pro-spectra $X,Y$ are \emph{weakly
equivalent} if there is an equivalence of  $\iota^{*} X $ and $\iota^*
Y$ in $\Pro(\Sp^{+})$.

We say that $X$ is \emph{weakly contractible} if it is weakly equivalent to a point. Finally,  we call $X \to Y \to Z$  a \emph{weak fibre sequence} if 
$\iota^{*}X \to \iota^{*}Y \to \iota^{*}Z$ is a fibre sequence in $\Pro(\Sp^{+})$.
\end{defn}

The functor $\pi_{i}\colon \Sp \to \Ab$ induces  functors
\[
\pi_{i}\colon \Pro(\Sp) \to \Pro(\Ab) \quad \text{and } \pi_{i}\colon \Pro(\Sp^{+}) \to \Pro(\Ab),
\]
given by $X = \lprolim{J}X_{j} \mapsto \pi_{i}(X) = \lprolim{J} \pi_{i}(X_{j})$, called the \emph{$i$-th homotopy pro-group}.
Abusing notation we write $\tau_{\leq n}\colon \Pro(\Sp) \to \Pro(\Sp)$ for the functor $\Pro(\tau_{\leq n})$.
\begin{lemma}\label{lem:description-of-weak-equivs}
For a map $f\colon X \to Y$ in $\Pro(\Sp)$ the following are equivalent:
\begin{enumerate}
\item $f$ is a weak equivalence.
\item $\tau_{\leq n}f$ is an equivalence for every $n\in \Z$.
\item There exists an $m\in \Z$ such that $\tau_{\leq m}f$ is an equivalence and $f$
  induces isomorphisms $\pi_n(X)\xrightarrow{\sim} \pi_n(Y)$ in $\Pro(\Ab)$ for all $n\in
  \mathbb Z$.
\end{enumerate}
\end{lemma}
\begin{proof}
Since the functor $\tau_{\leq n}$ factors through $\iota^{*}\colon \Pro(\Sp) \to \Pro(\Sp^{+})$, (1) implies (2).
Clearly, (2) implies (3). 
 It remains to show that (3) implies (1). 
Choose a level representation $\lprolim{I} X_{i} \to \lprolim{I} Y_{i}$. By assumption $\lprolim{I} \tau_{\leq m}X_{i} \to \lprolim{I} \tau_{\leq m}Y_{i}$ is an equivalence. 
Using the fibre sequence
\[
\Sigma^n H(\pi_{n}(X_{i})) \to  \tau_{\leq n}X_{i} \to \tau_{\leq n-1}X_{i}
\]
(similarly for $Y_{i}$) and that $\lprolim{I}\pi_{n}(X_{i}) \to \lprolim{I}\pi_{n}(Y_{i})$ is an isomorphism, we inductively show that 
we have equivalences $\lprolim{I} \tau_{\leq n}X_{i} \xrightarrow{\simeq} \lprolim{I} \tau_{\leq n} Y_{i}$ for all $n\in \Z$. Lemma~\ref{lem:explicit-description-of-i-star} now implies that $f$ is a weak equivalence. 
\end{proof}


\subsection{Pro-spaces}
\label{sec:pro-spaces}

Let $\Spaces^{0}_{*}$ denote the $\infty$-category of pointed connected spaces.

\begin{defn}
A morphism $X \to Y$  in $\Pro(\Spaces^{0}_{*})$ is called a \emph{weak equivalence} if it induces an isomorphism between homotopy pro-groups $\pi_{i}(X) \to \pi_{i}(Y)$ for every $i > 0$.
\end{defn}

Isaksen \cite{Isaksen-simplicial} has a general notion of weak equivalences between pro-simplicial sets. 
There is a canonical functor from the $\infty$-category associated to the category of pro-objects in pointed connected simplicial sets to $\Pro(\Spaces^{0}_{*})$ and this functor detects weak equivalences by \cite[Cor.~7.5]{Isaksen-simplicial}.

Let $f\colon X \to Y$ be a morphism between pro-objects of connected spectra and $\Omega^{\infty}f\colon \Omega^{\infty}X \to \Omega^{\infty}Y$ the induced map on pro-objects of infinite loop spaces. 

The following is clear from Lemma~\ref{lem:description-of-weak-equivs}.
\begin{lemma}
	\label{lemma:weak-equivs-infinite-loop-space}
The map $f$ is a weak equivalence in the sense of Definition~\ref{def:weak-equivalence} if and only if  $\Omega^{\infty}f$ is a weak equivalence in $\Pro(\Spaces_{*}^{0})$.
\end{lemma}

\begin{lemma}
	\label{lem:we-induces-homology-iso}
If $X \to Y$ is a weak equivalence in $\Pro(\Spaces_{*}^{0})$ then the induced map on  integral homology pro-groups $H_{i}(X) \to H_{i}(Y)$ is an isomorphism in $\Pro(\Ab)$.
\end{lemma}
This can be deduced from \cite[Thm.~7.3]{Isaksen-simplicial}.

There is also the following pro-version of the Whitehead theorem, see \cite[Cor.~4.3]{Singh}.
\begin{prop}
	\label{prop:pro-Whitehead}
Let $f\colon X \to Y$ be a morphism of pro-systems of pointed connected nilpotent spaces. 
If the induced map $f_{*}\colon H_{i}(X) \to H_{i}(Y)$ is an isomorphism in $\Pro(\Ab)$ for every $i > 0$, then $f$ is a weak equivalence in $\Pro(\Spaces_{*}^{0})$.
\end{prop}


\section{Non-archimedean rings}
	\label{sec:non-archimedean-rings}

\subsection{Adic rings}\label{sec.adicrin}

In this note we call a complete topological ring $A$ an \emph{adic ring} if a system of neighborhoods of zero is
given by the ideals $\{ (\pi^n)\, |\, n>0 \}$ for some element $\pi\in A$; so $A \cong \lim_{n} A/(\pi^{n})$. For an element
$\pi\in A$ with this property we also say that $A$ is \emph{$\pi$-adic}.

For such an adic ring $A$ and element $\pi\in A$ we denote by
\[
A\langle t_1 , \ldots , t_m \rangle = \lim_n A[t_1, \ldots , t_m ]/(\pi^n)
\]
the completion of the polynomial ring. Fix $a_0\in A$. By
\begin{equation}\label{comppolyadic}
  A\langle \Delta^m_{a_0} \rangle  = A\langle t_0 , \ldots , t_m \rangle /(t_0 + \cdots
  + t_m-a_0)
\end{equation}
we denote the standard $m$-simplex of `diameter $a_0$'. So $ A\langle \Delta^m_{a_0}
\rangle$ is the $\pi$-adic completion of 
\begin{equation}\label{comppolyadicpoly}
 A[ \Delta^m_{a_0} ]  = A[ t_0 , \ldots , t_m ] /(t_0 + \cdots
  + t_m-a_0).
\end{equation}
Note that there exists an $A$-algebra isomorphism
\[
A \langle x_1 , \ldots , x_m \rangle \xrightarrow{\simeq} A\langle \Delta^m_{a_0} \rangle
\quad \quad x_i \mapsto t_i .
\]
For varying $m$ and fixed $a_0$ the $ A\langle \Delta^m_{a_0} \rangle$ form in the usual
way a simplicial ring which we denote by $ A\langle \Delta_{a_0} \rangle$. If we choose
another element $a_1\in A$ we get a morphism of simplicial $A$-algebras
\[
\Psi_{a_1}^{a_0}\colon   A\langle \Delta_{a_1 a_0} \rangle \to  A\langle \Delta_{ a_0} \rangle
\]
by mapping each variable $t_i$ to $a_1 t_i$. The same applies to the polynomial rings $ A[ \Delta^m_{a_0} ]$.

In the following proposition we summarize some facts about excellent rings.
Let $\kappa$ be a complete discretely valued field with ring of integers $\kappa_0$. We
abbreviate $A\langle t_1 , \ldots , t_m \rangle$ by $A\langle \underline t \rangle$.

\begin{prop}\label{prop.adicexc} For a noetherian complete adic ring $A$ we have:
  \begin{itemize}
\item[(i)] For an element $\pi\in A$ as above the ring $A$ is quasi-excellent if and only
  if $A/(\pi)$ is quasi-excellent
\item[(ii)] If $A$ is quasi-excellent, then $A\langle \underline t \rangle$ is
  quasi-excellent, and the ring homomorphism $A\to A\langle \underline t \rangle$ is
  regular.
\item[(iii)] If $A$ is a $\kappa_0$-algebra of topologically  finite type, it is excellent.
  \end{itemize}
\end{prop}

\begin{proof}
Part~(i) is due to Gabber, see~\cite[Thm.~I.9.2]{illgab}. Part~(ii) follows immediately from
part (i) and~\cite[0AH2]{stacks-project}.
For part (iii) note that a
$\kappa_0$-algebra $A$ which is topologically of finite type is a quotient of a ring
$\kappa_0 \langle \underline t \rangle$, which is quasi-excellent (by (ii)) and which is
universally catenary as it is regular (again by (ii)).
\end{proof}

\subsection{Tate rings}\label{subsec:nonarch.tate}

For the convenience of the reader we recall some facts about Tate rings in this section,
see~\cite[Sec.~1]{Huber}.
Tate rings are generalizations of affinoid algebras and we are going to state most of our
results about  $K$-theory in the language of Tate rings.

\begin{defn}\label{defn.tatering}
A topological ring $A$ is called a {\em Tate ring} if there exists an open subring $A_0\subset
A$ which is a complete adic ring in the sense of Subsection~\ref{sec.adicrin}, i.e.\ it
has the $\pi$-adic topology for some $\pi\in A_0$ and $A=A_0[1/\pi]$. We call such a
subring $A_0$ a {\em ring of definition} of $A$.
A subset $B \subset A$ is called \emph{bounded} if it is contained in $\pi^{n}A_{0} \subset A$ for some integer $n$ and an element $a \in A$ is called \emph{power bounded} if the set $\{ a^{n} \,|\, n \geq 0\}$ is bounded.
\end{defn}

An element $\pi$ as in Definition~\ref{defn.tatering} is clearly invertible and
topologically nilpotent in $A$. 

\begin{rmk}\label{rmk.tateringofdef}
  The following observations about a Tate ring $A$ are easy to prove.
  \begin{itemize}
  \item[(i)]
For an invertible and topologically nilpotent $\pi \in A$ and a ring of definition
$A_0\subset A$
there exists $n>0$ such that $\pi^n\in A_0$ and
$A_0$ has the $\pi^n$-adic topology.
\item[(ii)]
  If $A_0$ and $A'_0$ are both rings of definition
of the Tate ring $A$ then $A_0\cap A'_0$ and $A_0 \cdot A'_0$ are also rings of definition
of $A$.
\item[(iii)] For a power bounded element $a_0\in A$ there exists a ring of definition
  $A_0\subset A$ with $a_0\in A_0$.
\end{itemize}
\end{rmk}

 Quite often in this note we need to assume the existence
of a `nice' ring of definition of a Tate ring $A$. We usually consider the  following
condition on  $A$:

\begin{itemize}
\item[$(\dagger)_A$] \phantomsection\label{condition.dagger} There exists a ring of definition $A_0\subset A$ which is noetherian
  and there exists a proper morphism $p\colon X\to \Spec (A_0)$ such that $X$ is regular and such that $p$ is
  an isomorphism over $\Spec (A)$.
\end{itemize}  
There is also the following weaker version of $(\dagger)_{A}$,  which is not
used in later sections, but is included here for completeness.
\begin{itemize}
\item[$(\ddagger)_A$] There exists a ring of definition $A_0\subset A$ which is noetherian
  and quasi-excellent.
\end{itemize}

\begin{rmk}
By~\cite[Sec.~IV.7.9]{EGA} condition $(\dagger)_A$ implies condition $(\ddagger)_A$. If $A$ is regular
and if one has resolution of singularities available, for example if  $\mathbb Q\subset
A_0$ \cite{temkin}, then condition  $(\ddagger)_A$ conversely implies condition $(\dagger)_A$.
\end{rmk}

We define the Tate algebra over the Tate ring $A$ by
\[
A \langle t_1 , \ldots ,  t_m \rangle =( A_0 \langle t_1 , \ldots ,  t_m \rangle )\otimes
_{A_0} A.
\]
It is easy to see that the Tate algebra does not depend on the choice of the ring of
definition $A_0$. For a power bounded element $a_0\in A$ we define the standard $m$-simplex with `diameter $a_0$' by the same
formula as in~\eqref{comppolyadic}, as well as its polynomial
version~\eqref{comppolyadicpoly}. 
Given another power bounded $a_1\in A$  there exists a well-defined map
\[
\Psi_{a_1}^{a_0}\colon   A\langle \Delta_{a_1 a_0} \rangle \to  A\langle \Delta_{ a_0} \rangle
\]
 mapping each variable $t_i$ to $a_1 t_i$.
In these constructions one uses Remark~\ref{rmk.tateringofdef}(iii) to reduce to the
analogous constructions for adic rings.

\begin{lem}
Consider a Tate ring $A$ which satisfies condition $(\dagger)_A$ (resp.\ $(\ddagger)_A$).
Then  $A\langle \underline t \rangle$  satisfies condition $(\dagger)_{A\langle \underline t \rangle}$
(resp.\ $(\ddagger)_{A\langle \underline t \rangle}$). Moreover the ring homomorphism $A\to A\langle \underline t \rangle$ is regular.
\end{lem}

Here we abbreviate $A\langle t_1 , \ldots , t_m \rangle$ by $A\langle \underline t
\rangle$ again.

\begin{proof}
This is immediately clear from Proposition~\ref{prop.adicexc}. Note that in the situation
of $(\dagger)_A$ the scheme $X\otimes_{A_0} A_0\langle \underline t
\rangle$ is
regular as the ring homomorphism $A_0\to  A_0\langle \underline t \rangle$ is regular
by~Proposition~\ref{prop.adicexc}(ii) and $X$ is regular.
\end{proof}

\begin{lem}
If $A$ is an affinoid algebra over a complete discretely valued field $\kappa$, then condition
$(\ddagger)_A$ holds. If moreover $A$ is regular and the residue field of $\kappa$ has
characteristic zero, then condition
$(\dagger)_A$ holds.
\end{lem}

\begin{proof}
For an affinoid algebra $A$ over $\kappa$ choose a  ring $A_0$ which is flat and topologically of finite type
over $\kappa_0$. Then $A_0$ is a ring of definition for $A$ and it is excellent by
Proposition~\ref{prop.adicexc}(iii). For $\mathbb Q\subset A_0$ we know that the scheme
$\Spec (A_0)$ has a resolution of singularities \cite{temkin}.
\end{proof}

\subsection{Normed rings}
	\label{sec:normed-rings}

When we study the Karoubi--Villamayor analog of our analytic $K$-theory in Section~\ref{sec.kv} it is convenient
to do it for a more general class of non-archimedean rings than Tate rings, namely normed
rings. This material is more or less well-known, but for the reader's convenience 
we recall some details in this subsection.

\begin{defn}
Let $A$ be an associative (not necessarily commutative or unital) ring.
A \emph{non-archimedean norm} on $A$ is a map $\norm\colon A \to \R_{\geq 0}$ satisfying
\begin{enumerate}
\item $\|x\|=0 \Leftrightarrow x=0$,
\item $\|-x\|=\|x\|$,
\item $\|x+y\| \leq \max\{\|x\|,\|y\|\}$,
\item $\|xy\| \leq \|x\|\|y\|$,
\end{enumerate}
for all $x,y\in A$.
\end{defn}
In the following, all norms  will be non-archimedean, and for simplicity we will usually drop the adjective `non-archimedean'.
A ring equipped with a norm is called a \emph{normed ring}. The normed ring $(A, \norm)$ is called \emph{complete} if $A$ is complete for the metric induced by the norm, and it is called \emph{unital} if $A$ is unital and $\|1\| = 1$.

\begin{ex}\phantomsection\label{ex:normed-rings}
\begin{itemize}
\item[(i)] Any associative ring $A$ can be equipped with the trivial norm $\norm$ given by $\|a\|=1$ if $a\not=0$, $\|0\|=0$.
\item[(ii)] If $(A,\norm)$ is a normed ring, one can equip its unitalization $\widetilde A = A \rtimes \Z$
with the norm 
\[
\|(a,n)\| = \max\{\|a\|, \|n\|_{\Z}\}
\]
where $\norm_{\Z}$ is the trivial norm on $\Z$. Then  $(\widetilde A, \norm)$ is a unital  normed ring and the inclusion $A \hookrightarrow \widetilde A$ is isometric. It is complete if $A$ was.

\item[(iii)] Let $A$ be a Tate ring with a ring of definition $A_0\subset A$ and an element
  $\pi\in A_0$ which is a topologically nilpotent unit in $A$.
 Fix $0<\epsilon <1$.
We define the gauge norm associated to $A_{0}, \pi$, and $\epsilon$ by 
\[
\|a\|_{A_{0},\pi,\epsilon} = \epsilon^{\max\{i\in\Z\,|\, a\in\pi^{i}A_{0}\}}.
\]
A different choice of $A_{0},\pi$, and $\epsilon$ gives an equivalent norm in the sense defined below.
\end{itemize}
\end{ex}

\begin{defn}
	\label{def.s-bounded}
A homomorphism between normed rings $\phi\colon (A,\norm_{A}) \to (B, \norm_{B})$ is called \emph{$s$-bounded} for some real number $s>0$ if there is a $C >0$ such that 
$\|\phi(x)\|_{B} \leq C\|x\|_{A}^{s}$ for all $x\in A$. It is called \emph{bounded} if it is $s$-bounded for some $s$.

Two norms $\norm, \norm'$ on a ring $A$ are called \emph{equivalent} if the identities $(A,\norm) \to (A,\norm')$ and $(A,\norm') \to (A,\norm)$ are bounded.
\end{defn}

Let $(A, \norm)$ be a complete normed ring, and let $\rho>0$ be a real number. 
\newcommand{\var}{t}
We define the ring of $\rho$-convergent power series with variables $\var_1,\ldots , \var_m$ and with coefficients in $A$
\[
A\langle \var_1, \ldots , \var_m\rangle_{\rho} = \{\sum_{I} a_{I} \var^{I} \,|\, \|a_{I}\|\rho^{|I|} \xrightarrow{|I|\to\infty} 0\},
\]
where we  use  standard multi-index notation.
The variables commute with every element of the ring.
We set
\[
\|\sum_{I} a_{I}\var^{I}\|_{\rho} = \max\{\|a_{I}\|\rho^{|I|}\}.
\]
This defines a norm $\norm_{\rho}$  on $A\langle \var\rangle_{\rho}$ and, equipped with this norm,
$A\langle \var\rangle_{\rho}$ becomes again a  complete normed ring. (See \cite[Ch.~1, 1.4.1]{BGR}
for the case of a commutative unital complete normed ring and radius of convergence
$\rho=1$. Everything carries over to our situation mutatis mutandis.)

Let $(A,\norm)$ be a complete normed ring and let $\rho \ge 1$. We define
\[
A\langle \Delta^{m}_{\rho}\rangle = A\langle \var_{0}, \dots, \var_{m}\rangle_{\rho}/I_{n,\rho}
\]
where $I_{n,\rho}$ is the twosided ideal 
\[
I_{n,\rho} = \{ \var_{0}f + \dots + \var_{m}f -f\,|\, f\in A\langle \var_{0}, \dots, \var_{n}\rangle_{\rho}\}.
\]
It is then clear that $A\langle \Delta_{\rho}\rangle$ with
the usual face and degeneracy maps is a simplicial ring.   For radii $\rho' \geq \rho\geq 1$ there is a natural restriction map
$A\langle \Delta_{\rho'}\rangle \to A\langle \Delta_{\rho}\rangle$ of simplicial $A$-algebras.

\begin{rmk}
	\label{rmk.Tatenormedsimplicial}
  For a Tate ring $A$ as in Example~\ref{ex:normed-rings}(iii) and for
   $\rho=\epsilon^{-j}$ with $j\ge 0$ there exists an isomorphism of simplicial $A$-algebras
   \[
     \Theta\colon A \langle \Delta_{\pi^j} \rangle \xrightarrow{\simeq}  A \langle
     \Delta_{\rho} \rangle, \quad \quad t_i \mapsto \pi^j \var_i.
   \]
For $\rho' = \epsilon^{-j-1} $ we obtain a commutative square of simplicial rings
\[
  \xymatrix{
A\langle \Delta_{\pi^{j+1}}\rangle \ar[r] ^{\Psi^{\pi^j}_\pi}  \ar[d]_\Theta^\simeq & A\langle \Delta_{\pi^j}\rangle  \ar[d]^\Theta_\simeq    \\
A\langle \Delta_{\rho'}\rangle \ar[r] & A\langle \Delta_{\rho}\rangle.
}
\]
\end{rmk}


\section{Algebraic \textit{K}-theory of non-archimedean rings} 
	\label{sec:K-theory-of-na-rings}

In this section we summarize some background material on algebraic $K$-theory and we
study algebraic $K$-theory of non-archimedean rings.

\subsection{\textit{K}-theory spectra and the plus-construction}
\label{sec:K-theory-spectra-and-plus-construction}

Let $R$ be a unital ring.
We denote by 
\[
k(R) \in \Sp_{\geq 0}
\] 
the connective algebraic $K$-theory spectrum associated to the exact category of finitely generated projective (right) $R$-modules.
Later we will also use the 
non-connective algebraic $K$-theory spectrum $K(R)$.
Its connected cover $\tau_{\geq 1} k(R)$ is denoted by $\KGL(R)$.
We have a natural equivalence
\[
\BGL(R)^{+} \simeq \Omega^{\infty}\KGL(R)
\]
where the left-hand side is Quillen's plus construction.

Now let $R$ be a simplicial ring.

\begin{defn}
We write $k(R)$ for the \emph{degree-wise $K$-theory} of $R$, i.e.~the geometric realization of the simplicial spectrum $[p] \mapsto k(R_{p})$. Similarly, we write $\KGL(R)$ for the geometric realization of the simplicial spectrum $[p] \mapsto \KGL(R_{p})$.
\end{defn}

Note that in general $\KGL(R)$ is  not the connected cover of $k(R)$. The infinite loop space  $\Omega^{\infty}\KGL(R)$ can again be described as a plus construction.
We write
$\BGL(R)$ for the diagonal of the bisimplicial set $[p]\mapsto \BGL(R_{p})$. Note that in general the diagonal is weakly equivalent to the geometric realization of a bisimplicial set.
Let $(-)^{+}\colon \sSet \to \sSet$ be the functorial plus construction as defined in \cite[Def.~3.1.1.5]{Dundas}. 

\begin{lem}
	\label{lem:plus-simplicial}
There are natural weak equivalences
\[
\BGL(R)^{+} \simeq    \diag([p]\mapsto \BGL(R_{p})^{+})
\]
and 
\[
\BGL(R)^{+} \simeq \Omega^{\infty} \KGL(R).
\]
\end{lem}

\begin{proof}
The first equivalence follows from \cite[Lemma 3.1.3.1]{Dundas}. It implies the second one, since geometric realization commutes with $\Omega^{\infty}$ for connective spectra \cite[Prop.~1.4.3.9]{HA}.
\end{proof}

This can be simplified even more for connected simplicial rings,  i.e.~simplicial rings whose underlying simplicial set is connected:
\begin{lem}
	\label{lem.KGLinfiniteloops}
Assume that the simplicial ring $R$ is connected. Then there is a natural weak equivalence
\[
\BGL(R) \simeq \Omega^{\infty} \KGL(R).
\]
\end{lem}

\begin{proof}
Since $R$ is connected, every elementary matrix in $\GL(R_{0})$ can be connected with the identity matrix by a $1$-simplex in $\GL(R)$. This implies that $\pi_{1}\BGL(R) \cong \pi_{0}\GL(R)$ is abelian, and hence $\BGL(R) \to \BGL(R)^{+}$ is a weak equivalence. Now the claim follows from Lemma~\ref{lem:plus-simplicial}.
\end{proof}

\subsection{Pro\texorpdfstring{-$\mathbb A^1$-}{-}homotopy invariance for  algebraic  \textit{K}-theory}
\label{subsec.alghomo}

The observations of this subsection resulted from a discussion with M.~Morrow.
Let $R$ be a commutative noetherian ring and let $\pi\in R$. Let $\Sch_R$ be the category of schemes
of finite type over $R$. In this subsection we study the $K$-theoretic properties of the
pro-system of $R$-algebras $\lprolim{n\in \mathbb N} S_n$ with $S_n$ the polynomial ring
$R[\underline t] = R[t_1,\ldots, t_m]$ for all $n\in \mathbb N$ and with transition maps $S_{n+1}\to S_n$  characterized
by $t_i\mapsto \pi t_i$ ($1\le i\le m$). We shall abbreviate this pro-system also as
$\lprolim{\underline t\mapsto \pi \underline t} R[\underline t]$.

Let $F$ be a contravariant functor from $\Sch_R$ to $\Pro(\Sp)$. We define a new functor
$N^\pi F$ from $\Sch_R $ to $\Pro(\Sp)$ by
\begin{equation}\label{eqdefN}
N^\pi F(X) = \cofib ( F(X) \xrightarrow{p^*} \prolim{\underline t\mapsto \pi \underline t}
F(X\otimes_R R[\underline t] )),
\end{equation}
where $p:X\otimes_R R[\underline t] \to X$ is the canonical map. Note that this $p^*$ has a left
inverse given by the zero section.

We are interested in criteria, which guarantee that $N^\pi F(X)$ is contractible resp.\
weakly contractible, see Definition~\ref{def:weak-equivalence}.
For $X$ in $\Sch_R$ we write $X_n$ for the scheme $X\otimes_R R/(\pi^n)$.

\begin{lem}\label{lem.Xncontr}
Let $X$ be in $\Sch_R$ such that there exists $n>0$ with $X=X_n$. Then $N^\pi
F(X)$ is contractible.
\end{lem}

\begin{proof}
This is clear as the map $X\otimes_R R[\underline t]$ given by $t\mapsto \pi^n t$ factors through $X$.
\end{proof}

A morphism $f:\tilde X\to X$ in $\Sch_R$ is called \emph{admissible} if $f$ is proper and an
isomorphism over $X\otimes_R R[1/\pi]$.
Consider the following conditions on the functor $F$:

\begin{itemize}
\item[$({\rm P1})_F$] For any admissible morphism $f:\tilde X\to X$ the functor $F$ satisfies
  pro-descent, i.e.\ we have a weak equivalence
  \[
    \prolim{n} F(X,X_n) \to \prolim{n} F(\tilde X, \tilde X_n).
  \]
\item[$({\rm P2})_F$] For any regular scheme $X$ in $\Sch_R$ the pro-spectrum $N^\pi F(X)$ is
  weakly contractible.
\end{itemize}
Here for a morphism of schemes $Y' \to Y$ in $\Sch_R$ we use the notation
\[
  F(Y,Y') = \fib( F(Y) \to F(Y')).
\]

\begin{lem}\label{lem.Nindep}
 Let $F$ be a functor satisfying  condition $({\rm P1})_F$ and let $p:\tilde X \to X$ be an
 admissible morphism in $\Sch_R$. Then $p$ induces a weak equivalence
 \[
N^\pi F(X) \xrightarrow{\simeq} N^\pi F (\tilde X).
 \]
\end{lem}

\begin{proof}
  Consider the commutative diagram of fibre sequences of pro-spectra
  \[
    \xymatrix{
\prolim{n} N^\pi  F(X,X_n) \ar[d]_{\simeq} \ar[r]^-{\simeq} & N^\pi F(X) \ar[d] \ar[r] & \prolim{n}  N^\pi
F(X_n) \ar[d] \\
\prolim{n} N^\pi  F(\tilde X,\tilde X_n) \ar[r]^-{\simeq}  & N^\pi F(\tilde X) \ar[r] & \prolim{n}
N^\pi F(\tilde X_n)
    }
  \]
The pro-spectra on the right are both contractible by Lemma~\ref{lem.Xncontr}, so the left horizontal maps are
equivalences. The left vertical map is a weak equivalence by condition $({\rm P1})_F$.
\end{proof}

\begin{prop}\label{khpro.prop} Let $F$ be a functor satisfying the conditions $({\rm P1})_F$ and $({\rm P2})_F$. Then
  for any $X$ in $\Sch_R$ such that there exists an admissible morphism $\tilde X\to X$
  with $\tilde X$ regular the pro-spectrum $N^\pi F(X)$ is weakly contractible.
\end{prop}

\begin{proof}
  By Lemma~\ref{lem.Nindep} the map $N^\pi F(X) \to N^\pi F(\tilde X) $ is a weak
  equivalence of pro-spectra. As $N^\pi F(\tilde X)$ is weakly contractible by $({\rm P2})_F$
  the same is true for $N^\pi F(X) $.
\end{proof}

Using that non-connective algebraic $K$-theory spectrum $K$ satisfies the conditions
$({\rm P1})_F$ resp.\  $({\rm P2})_F$ by \cite[Thm.~A]{KST-Weibel} resp.\ \cite[Thm.~V.6.3]{K-book}  we get:

\begin{cor}\label{khpro.cor}
For $X$ as in Proposition~\ref{khpro.prop} we get a weak equivalence of pro-spectra
\[K(X) \xrightarrow{\simeq} \prolim{\underline t\mapsto \pi \underline t} K(X\otimes_R
  R[\underline t]).\]
\end{cor}

We finish this subsection with a variant of the above observations. Let now $R$ be a commutative not
necessarily noetherian ring and $\pi\in R$. For a covariant functor $F$ from $R$-algebras to $\Pro
(\Sp )$ we define by the same method as in~\eqref{eqdefN} a new functor $N^\pi F$ from
$R$-algebras to $\Pro (\Sp )$.

We consider the following excision property for $F$:
\begin{itemize}
\item[$({\rm E})_F$] For any extension of $\pi$-torsion free $R$-algebras $S_1\subset S_2$  such that there exists
  $m>0$ with $\pi^m S_2\subset S_1$ we have a weak equivalence
  \[
\prolim{n} F(S_1,S_1/(\pi^n)) \xrightarrow{\simeq} \prolim{n} F(S_2,S_2/(\pi^n)). 
  \]
\end{itemize}

The same argument as in the proof of Lemma~\ref{lem.Nindep}, replacing condition $({\rm P1})_F$
by condition $({\rm E})_F$  gives:

\begin{lem}\label{lem.exci}
Let $F$ satisfy condition $({\rm E})_F$ and let $S_1\subset S_2$ be an extension as in condition
$({\rm E})_F$. Then the map of pro-spectra
\[
N^\pi F(S_1) \xrightarrow{\simeq} N^\pi F (S_2)
\]
is a weak equivalence.
\end{lem}

\begin{rmk}\phantomsection\label{rmk.lem.exci}
  \begin{itemize}
    \item[(i)]
Algebraic $K$-theory satisfies the condition $({\rm E})_{K}$ by work of Suslin,
Geisser--Hesselholt and Morrow, see~\cite[Thm.~0.2]{Morrow}. Using the Main Theorem of \cite{LandTamme}, it is in fact possible to show that the map in condition $({\rm E})_{K}$ is an equivalence of pro-spectra. Since we only work with weak equivalences anyway, we refrain from writing out the details.
\item[(ii)] Combining (i) and
Lemma~\ref{lem.exci} we get a weak equivalence of pro-spectra
\[
N^\pi K(S_1) \xrightarrow{\simeq} N^\pi K (S_2).
\]
\end{itemize}
\end{rmk}

\subsection{\texorpdfstring{$K_0$}{K0} and \texorpdfstring{$K_1$}{K1} of non-archimedean rings}\label{algk.subKlow}

In this subsection we discuss some results about $K_0$ and $K_1$ of non-archimedean rings.

We define a continuous version of
$K_1$ of a Tate ring $A$ as the pro-group
\[
\Kco_1 (A) = \prolim{n} K_1(A)/(1+\pi^n A_0).
\]
Here $A_0$ is a ring of definition and $\pi\in A_0$ is a topologically nilpotent unit in $A$.
Following Morrow~\cite{morrow-hist} we shall in Section~\ref{sec.contK} give a definition of the continuous $K$-theory pro-group $\Kco_i(A)$ for
any $i$. Note that
the sequence
\[
0\to 1+\pi^n A_0 \to K_1(A_0) \to K_1(A_0 / (\pi^n)) \to 0
\]
is exact \cite[Lem.~III.2.4]{K-book}.

Let $F$ be a covariant functor from the category of Tate rings to an abelian category. Using a Bass type construction, compare~\cite[Sec.~III.4]{K-book}, we define a new functor
$\Sigma F$ on the same categories by
\[
\Sigma F(A)= \coker ( F(A\langle t \rangle ) \times F(A\langle t^{-1} \rangle ) 
\to  F(A\langle t,t^{-1} \rangle)).
\]
\begin{lem}\label{algK.lamdense}
For any Tate ring $A$ the natural map $\Sigma K_1 (A)\to \Sigma \Kco_1(A)$ is an
isomorphism of pro-groups.
\end{lem}

\begin{proof}
  We are going to show that the product map
  \begin{equation*}\tag{$\check C_n$}
  (1+\pi^n A_0\langle t \rangle)\times  (1+\pi^n A_0\langle t^{-1} \rangle) \to  (1+\pi^n
  A_0\langle t,t^{-1} \rangle)
  \end{equation*}
  is surjective for all $n>0$. The cokernel of~$\check C_n$ is the \v Cech cohomology
  group $H^1(\sU,1+\pi^n \cO)$, where $\sU$ is the standard covering of the formal $\pi$-adic scheme
  $(\mathbb P^1_{A_0})^\wedge$ and $\cO$ denotes its structure sheaf.
  Together with the vanishing of $H^1(\sU, \cO/\pi\cO)\cong H^1(\mathbb P^1_{A_0/(\pi)}, \cO_{\mathbb P^1_{A_0/(\pi)}})$, see~\cite[Prop.~III.2.1.12]{EGA},
the long exact
  sequence of \v Cech
  cohomology corresponding to the exact sequence
  \[
0\to 1+\pi^{n+1} \cO \to 1+\pi^n \cO \to \cO/\pi\cO \to 0
\]
 implies that the map
\[
\coker( \check C_{n+1}) \to \coker( \check C_n)
\]
is surjective. A simple convergence argument shows  that     $\lim_n
\coker( \check C_n)$ vanishes. Indeed,  consider an element
  $(\overline x_n)_n $ in this limit with representing elements $x_n \in 1+\pi^n
  A_0\langle t,t^{-1} \rangle$. This means there exists  elements $y_n$ in the domain of
  $\check C_n$ with $\check C_n(y_n)=x_n/x_{n+1}$. The product $z_n= y_n y_{n+1} y_{n+2}
  \cdots$ converges and satisfies $C_n(z_n)=x_n$, so $\overline x_n=0$.
\end{proof}

 In analogy with \cite[p.~77]{Gruson} we define
a canonical homomorphism \[s:K_0(A) \to K_1(C),\quad\quad [P] \mapsto [P\otimes_A
  C\xrightarrow{t} P\otimes_A C  ]. \]
Here and in the following we let $C=A\langle t,t^{-1} \rangle$.

\begin{prop}\label{algk.bassfund}
  For any Tate ring $A$
  the map $s$ induces an isomorphism $\bar s:K_0(A) \xrightarrow{\sim} \Sigma K_1(A)$.
\end{prop}

\begin{proof}
Gruson constructs a left inverse $r$ to $\bar s$, see \cite[Sec. IV.4]{Gruson}. This map $r$ has
the following geometric description: The adic space $(\mathbb P^1_A)^\an$ has the open covering
$\mathcal U = \{ \Spc A\langle t\rangle  , \Spc A\langle t^{-1}\rangle  \}$. The set of
analytic vector bundles up to isomorphism   ${\rm Vec}( (\mathbb P^1_A)^\an)$ is in
bijection with the set of algebraic vector bundles up to
isomorphism ${\rm Vec}( \mathbb P^1_A)$  by   GAGA, see \cite[Satz 5.1]{Koe} for the
  affinoid case.

Consider an
element $\xi \in \Sigma K_1(A)$ induced by a matrix $M\in \GL_n(C)$. One defines $r
(\xi )$ as the composition  of the maps  
\begin{align*}
\GL_n (C)& \to H^1(\sU,\GL_n) \to {\rm Vec}( (\mathbb P^1_A)^\an) \cong {\rm Vec}( \mathbb
  P^1_A) \to \\
 & K_0( \mathbb P^1_A) \xleftarrow[\simeq]{(p^*,([\cO(1)]-[\cO(0)])\cdot p^* )} K_0(A)
   \times  K_0(A) \xrightarrow{{\rm pr}_2} K_0(A),
\end{align*}
where ${\rm pr}_2$ is the projection to the second component.

In order to finish the proof of Proposition~\ref{algk.bassfund} we thus have to show that
$\bar s$ is surjective. The argument is parallel to the proof of \cite[Prop.~IV.6]{Gruson}. We
start with an element $\xi \in \Sigma K_1(A)$ represented by a matrix $M\in \GL_n(C)$.
Using the density of Laurent polynomials in $C$ and Lemma~\ref{algK.lamdense}, we can
assume without loss of generality that the entries of $M$ lie in $A[t,t^{-1}]$. For any integer
$N>0$ the matrix $t^N M$ represents the element $\xi + \bar s (A^{n N})\in \Sigma K_1(A)$. So after making
this replacement, we can assume without loss of generality that the entries of $M$
lie in $A[t]$.

As explained in parts c.\ and d.\ of the proof of~\cite[Prop.~IV.6]{Gruson} we can furthermore
reduce to the case that $M=b_0 + b_1 t$ with $b_0,b_1\in \mathrm{M}(n\times n; A)$ and
$b_0+b_1=1$.

By~\cite[Lem.~IV.6]{Gruson} there exists a decomposition of $A$-modules $A^n=P_0\oplus Q_0$,
which is stable under the action of $b_0$ and $b_1$ and such that $b_1$ is topologically
nilpotent on $P_0$ and $b_0$ is topologically nilpotent on $Q_0$. Let $P=P_0\otimes_A C$,
$Q=Q_0 \otimes_A C $ and let $M'$ be the matrix which acts by $M$ on $P$ and by $t^{-1}M$ on
$Q$. Then the element induced by $M'|_P$ in $K_1(C)$ comes from $K_1(A\langle t\rangle )$
and the element induced by $M'|_Q$ in   $K_1(C)$ comes from $K_1(A\langle t^{-1}\rangle
)$, so $M'$ represents zero in $\Sigma K_1 (A)$. On the other hand $M'$ represents the
element $\xi - \bar s ( [Q_0] )\in \Sigma K_1 (A)$.
\end{proof}

\begin{rmk}
The non-archimedean negative $K$-groups $\Sigma^i K_0(A)$ are studied
by Calvo~\cite{calvo}, she denotes them $K^{\rm top}_{-i}(A)$. Karoubi established a result analogous to
Proposition~\ref{algk.bassfund} using  a different convergence condition  in~\cite{Karoubi1971}.
\end{rmk}

\subsection{Analytic Bass delooping}\label{sec.bassdel}

In this subsection we study an abstract spectral version of the analytic Bass 
construction which we used in Proposition~\ref{algk.bassfund}.

Consider a functor $E$ from adic rings or from Tate rings to $\Pro (\Sp)$.
We define
\[
\Gamma E (A) = E(A\langle t\rangle) \sqcup_{E(A)} E(A\langle t^{-1}\rangle)
\]
and 
\begin{equation*}
\Lambda E(A) := \fib(\Gamma E(A) \to E(A\langle t,t^{-1}\rangle)).
\end{equation*}

Similarly, for a functor $F$  from adic rings or from Tate rings to an abelian category  we write
\[ 
\Gamma F (A) :=\coker\big(F(A) \to F(\At) \oplus F(\Ati)\big).
\]
For $E$ as above we then have 
\[ 
\pi_i(\Gamma E(A) ) \cong \Gamma(\pi_i E)(A)
\]
and  a long exact sequence
\begin{equation}\label{LE.longexactsequence}
\dots\to \Gamma(\pi_{i+1} E)(A) \to \pi_{i+1} (E(\Atti)) \to \pi_i (\Lambda E(A)) \to \Gamma(\pi_{i}E)(A) \to \dots
\end{equation}

Connective algebraic $K$-theory refines to a functor from commutative rings to $E_{\infty}$-ring spectra, and we can consider functors $E$ as above which carry a $K$-module structure. Fix a map from the sphere spectrum to $\Omega K(\Z[t,t^{-1}])$ representing the class of $t$ in $K_{1}(\Z[t,t^{-1}])$.
If $E$ carries a $K$-module structure, we use that map to construct a map $\lambda\colon E \to \Lambda E$ given on an adic or Tate ring $A$ by the composition
\[
\lambda \colon E(A) \xrightarrow{\mycup t} \Omega E(A \langle t, t^{-1} \rangle) \xrightarrow{\partial} \Lambda E(A),
\]
where the second map comes from the definition of $\Lambda E(A)$ as a fibre.

\begin{prop}[Bass fundamental theorem] \label{prop.absbass}
	\label{prop:Bass-fundamental-theorem}
Let $E$ be a functor as above, and assume that $\lambda\colon E(A) \to \Lambda E(A)$ is a weak equivalence for every $A$. Then we have exact sequences\vspace{-3ex}%
\begin{small}
\[
\xymatrix@1@C-0.2cm{
0 \ar[r]&  E_{i}(A) \ar[r]&  E_{i}(A\langle t\rangle) \oplus E_{i}(A\langle t^{-1}\rangle) \ar[r] & E_{i}(A\langle t,t^{-1}\rangle) \ar[r] & E_{i-1}(A) \ar@/_1.3pc/[l]_{\mycup t} \ar[r] & 0
}
\]
\end{small}%
where the split is induced by cup-product with $t$ and, as usual, $E_{i}(A) =
\pi_{i}(E(A))$, etc. In particular, there is an isomorphism $\lambda: E_i (A) \to \Sigma
E_{i+1}(A)$ of pro-groups.
\end{prop}

Here $\Sigma$ is as defined in Subsection~\ref{algk.subKlow}.

\begin{proof}
Since the map $A \to A\langle t\rangle$ is split, we get short exact sequences
\begin{equation}\label{seq:111}
0 \to E_{i}(A) \to E_{i}(A\langle t\rangle) \oplus E_{i}(A\langle t^{-1}\rangle) \to \Gamma E_{i}(A) \to 0.
\end{equation}
On the other hand, we have a long exact sequence
\[
\dots \to \Lambda E_{i}(A) \to \Gamma E_{i}(A) \to E_{i}(A\langle t,t^{-1}\rangle) \xrightarrow{\partial} \Lambda E_{i-1}(A) \to \dots 
\]
By assumption, the composition $E_{i-1}(A) \xrightarrow{\mycup t} E_{i}(A\langle t,t^{-1}\rangle) \xrightarrow{\partial} \Lambda E_{i-1}(A)$ is an isomorphism.
Hence the long exact sequence yields short exact sequences
\begin{equation}\label{seq:222}
0 \to \Gamma E_{i}(A) \to E_{i}(A\langle t,t^{-1}\rangle) \xrightarrow{\partial}  E_{i-1}(A) \to 0
\end{equation}
split by $\mycup t$.
Splicing the exact sequences \eqref{seq:111} and \eqref{seq:222} together gives the result.
\end{proof}

In the construction of non-connective analytic $K$-theory in Section~\ref{sec.anK} we will use the following analytic version of the Bass construction:
\begin{defn}
	\label{lem:bassification}
For a functor $E$ from adic or Tate rings to $\Pro(\Sp)$ which carries a module structure over connective algebraic $K$-theory we define its \emph{analytic Bass construction} by the formula
\begin{equation}\label{def;EB}
E^{B}(A) \simeq \colim (E(A) \xrightarrow{\lambda} \Lambda E(A) \xrightarrow{\Lambda(\lambda)} \Lambda^{2} E(A) \to \dots ).
\end{equation}
\end{defn}
Notice that we do not claim that $E^{B}$ satisfies the Bass fundamental theorem.
In Section~\ref{sec.anK} we will use the following simple observation.

\begin{lem}
	\label{lem:Bannihilates}
Assume that there is an integer $N$ such that $E_{i}(A) = 0$ for every $i > N$ and every
adic resp.\ Tate ring $A$ and that for any $A$ the pro-spectrum $E(A)$ is constant, i.e.\ is given by an ordinary
spectrum.
Then $E^{B}(A) \simeq 0$ for every such $A$.
\end{lem}

\begin{proof}
From the long exact sequence \eqref{LE.longexactsequence} we deduce that also
$\Lambda E_{i}(A) = 0$ for $i > N$ and every $A$. Iterating this argument, we see that $\Lambda^{n}E_{i}(A) = 0$ for $i>N$ and every $n$ and $A$.
Now let $i$ be any integer. Since $E$ has values in spectra, we have
\[
E^{B}_{i}(A) \simeq \colim_{k} \Lambda^{k}E_{i}(A).
\]
By construction, the $n$-fold composition $\Lambda^{k}E_{i}(A) \to \Lambda^{k+n}E_{i}(A)$ of transition maps  factors through $\Lambda^{k}E_{i+n}(A\langle t_{1}^{\pm1}, \dots, t_{n}^{\pm 1}\rangle)$. By the above, the latter group vanishes for $i+n>N$. Hence the colimit vanishes.
\end{proof}


\section{Continuous \textit{K}-theory}\label{sec.contK}

\subsection{Construction}

Let $A_0$ be a complete $\pi$-adic ring for some $\pi\in A_0$.

\begin{defn}
  Continuous $K$-theory of $A_0$ is defined as the pro-spectrum
  \[
    \Kco (A_0) = \prolim{n} K(A_0 / (\pi^n ) ).
  \]
\end{defn}

Clearly,  $\Kco (A_0)$ is independent of the choice of $\pi$ up to canonical equivalence of pro-spectra.

\begin{lem}\label{lem.KcoA0low}
  We have
  \[
    \Kco_i(A_0) = \begin{cases} \prolim{n} K_1(A_0)/(1+\pi^n A_0), & \text{ for } i=1, \\
       K_{0}(A_{0}) = K_{0}(A_{0} / (\pi) ), & \text{ for } i = 0, \\
       K_i(A_0/(\pi)), & \text{ for } i < 0.   
       \end{cases}
  \]
\end{lem}

\begin{proof}
  The first line is~\cite[Lem.~III.2.4]{K-book}, the second and third lines are a consequence of~\cite[Lemma~II.2.2]{K-book}.
\end{proof}

For a topological ring $R$ whose topology is generated by the powers of some ideal $I\subset
R$ we write
\begin{equation}
	\label{eq.def.Kbar}
\bar K (R) = K (R \text{ on } I)
\end{equation}
for the non-connective $K$-theory spectrum with support on $I$, see~\cite[Def.~6.4]{thomason}. By $\bark(R)=\tau_{\geq 0}\bar K(R)$ we denote the connective cover of $\bar K (R)$.

Let now $A$ be a Tate ring with a ring of definition $A_0\subset A$ and let $\pi\in A_0$ be a
topologically nilpotent unit in $A$.

\begin{defn}
	\label{def.cont.K}
  Continuous $K$-theory of $A$ (with respect to the ring of definition $A_0$) is defined as the pro-spectrum
  \[
    \Kco (A;A_0) = \cofib (\bar K(A_0 ) \to  \Kco(A_0 )).
  \]
  We also set
  \[
    \Kco (A) = \colim_{A_0}   \Kco (A;A_0),
  \]
  where the  colimit is over the partially ordered set of rings of definition $A_0\subset A$.
\end{defn}

Note that for a  Tate ring $A$  the partially ordered set of rings of
definition $A_0\subset A$ is directed, see  Remark~\ref{rmk.tateringofdef}(ii).

We are going to show in Proposition~\ref{prop.welldefinedcont}, following~\cite{morrow-hist}, that
\[
\Kco (A;A_0)\simeq \Kco (A).
\]
For this we need some preparations.
By the localization theorem~\cite[Thm.~7.4]{thomason} we have a fibre sequence of spectra
\begin{equation}\label{locseqA0}
\bar K(A_0 ) \to K(A_0 ) \to K(A ),
\end{equation}
so there exists a canonical morphism $K(A) \to \Kco (A)$. Moreover, the pro-spectrum  $\Kco(A; A_0)$ is part of
a cartesian square
\begin{equation}\label{squar.carcompcon}
  \xymatrix{
    K(A_0) \ar[r] \ar[d] & K(A)  \ar[d] \\
    \Kco (A_0) \ar[r]  &  \Kco (A; A_0) .
  }
\end{equation}

\begin{prop}\label{prop.welldefinedcont}
For a Tate ring $A$ with ring of definition $A_0$ the map $ \Kco (A; A_0) \xrightarrow{\sim} \Kco (A)$ is a weak equivalence.
\end{prop}

\begin{proof}
   By Remark~\ref{rmk.tateringofdef} it is sufficient to show that for $A_0\subset A'_0$ two rings of
  definition of $A$ the map $\Kco (A;A_0) \to \Kco (A ;A'_0 )$ is a weak equivalence of
  pro-spectra. For this consider the commutative cube of pro-spectra
\begin{equation}\label{square.kcomprdef}
  \xymatrix@R=2.5ex@C=1.3ex{ 
  K(A_0) \ar[dd] \ar[dr] \ar[rr] &&  K(A) \ar@{=}[dr] \ar'[d][dd]  &  \\
&  K(A'_0 ) \ar[dd] \ar[rr] &&    \ar[dd]  K(A)  \\
\Kco (A_0)  \ar[dr] \ar'[r][rr]  & & \Kco (A;A_0) \ar[dr]^\alpha  &  \\
&  \Kco (A'_0)  \ar[rr] &&  \Kco (A;A'_0).
}
\end{equation}
The back square is the same as the cartesian square~\eqref{squar.carcompcon} and the
front square is the corresponding cartesian square with $A_0$ replaced by $A'_0$. The left square is weakly
cartesian by Remark~\ref{rmk.lem.exci}(i). So the right square is weakly cartesian as well or in other words $\alpha$ is a
weak equivalence of pro-spectra.
\end{proof}

In Subsection~\ref{algk.subKlow} we gave another definition of the pro-group $\Kco_1(A)$
for a Tate ring $A$. In fact both definitions agree.

\begin{lem}\label{lem.calK1cont}
  The natural map of pro-groups
  \[
   \prolim{n} K_1(A)/(1+\pi^n A_0)  \xrightarrow{\simeq} \Kco_1 (A)
  \]
  is an isomorphism.
\end{lem}

\begin{proof}
  This follows from the five lemma and the following commutative diagram with exact columns.
  \[
    \xymatrix{
      \bar K_1(A_0) \ar[d] \ar@{=}[r] & \bar K_1(A_0) \ar[d] \\
      \prolim{n} K_1(A_0)/(1+\pi^n A_0) \ar[d] \ar[r]^-\simeq  & \Kco_1(A_0) \ar[d] \\
      \prolim{n}
      K_1(A)/(1+\pi^n A_0) \ar[d] \ar[r]&  \Kco_1(A) \ar[d]\\
      \bar K_0(A_0) \ar[d]  \ar@{=}[r]& \bar K_0(A_0)   \ar[d]\\
      K_0(A_0) \ar[r]^-\simeq &  \Kco_0(A_0)  
    }
  \]
Here the left column is obtained from the long exact sequence of homotopy groups of the fibre sequence \eqref{locseqA0} by modding out by $1+\pi^{n}A_{0}$ and passing to the pro-system; the right column is obtained from the defining fibre sequence $\bar K(A_{0}) \to \Kco(A_{0}) \to \Kco(A;A_{0})$ and the weak equivalence of Proposition~\ref{prop.welldefinedcont}.
The second and the fifth horizontal arrows are isomorphisms by Lem\-ma~\ref{lem.KcoA0low}, hence so is the third one.
\end{proof}

An argument parallel to the proof of Lemma~\ref{lem.calK1cont}, using Theorems~2.5 and~2.13
from~\cite{stein}, shows:

\begin{lem}\label{lem.k2cont}
  For a Tate ring which has a local ring of definition $A_0\subset A$ the  map
  \[
    \prolim{n} K_2(A)/\{ 1+\pi^n A_0 , A_0^\times \} \xrightarrow{\simeq}   \Kco_2 (A)
  \]
  is an isomorphism of pro-groups.
\end{lem}

\begin{ex}
Assume that the Tate ring $A$ admits a ring of definition $A_{0}$ which is noetherian and regular. 
Then $\bar K_{i}(A_{0}) = 0$ for $i<0$. By definition of $\Kco(A)$ and Lemma~\ref{lem.KcoA0low} we thus get isomorphisms
\[
K_{i}(A_{0}/(\pi)) \xrightarrow{\simeq} \Kco_{i}(A) \quad \text{ for } i<0.
\]
In particular, the negative continuous $K$-groups need not vanish even if $A$ admits a regular ring of definition.

A concrete example is the following. Let $\kappa$ be a discretely valued field with ring of integers $\kappa_{0}$ and uniformizer $\pi$. Set $A_{0} = \kappa_{0} \langle x,y \rangle/(y^{2} - x^{3} + x^{2} - \pi)$ and $A = A_{0} [1/\pi]$, a thickened version of a nodal curve. Then $A_{0}$ is regular, but $\Kco_{-1}(A) \cong K_{-1}(A_{0}/(\pi)) \cong \Z$ (see \cite[Exc.~III.4.12]{K-book}).
\end{ex}

 Continuous $K$-theory of a Tate ring $A$ can also
be calculated in terms of a Raynaud type geometric model of $A$. By this we mean the
following: Assume that there exists a noetherian ring of definition $A_0\subset A$ and let
$\pi\in A_0$ be a topologically nilpotent unit in $A$. Assume given an ``admissible morphism'' or more precisely a proper
morphism $p\colon X\to  \Spec (A_0)$ which is an isomorphism over $\Spec (A)$. For $n>0$ set
$X_n=X\otimes_{A_0} A_0/(\pi^n)$ and define continuous $K$-theory of $X$ as the pro-spectrum
\[
\Kco (X) = \prolim{n} K(X_n).
\]

\begin{prop}\label{prop.Kcoray}
  There is a weak fibre sequence of pro-spectra.
  \[
    \bar K(X) \to \Kco (X)  \to \Kco (A).
  \]
\end{prop}
Here $\bar K(X)$ stands for the algebraic $K$-theory spectrum of $X$ with support in $X_1$.

\begin{proof}
The proof is very similar to the proof of Proposition~\ref{prop.welldefinedcont}.
 Let for the moment $W$ be the cofibre of $ \bar K(X) \to \Kco (X)$.
Then we get a commutative cube of pro-spectra 
\begin{equation}\label{squar.carcompcon2}
  \xymatrix@R=2.5ex@C=1.3ex{ 
  K(A_0) \ar[dd] \ar[dr] \ar[rr] &&  K(A) \ar@{=}[dr] \ar'[d][dd]  &  \\
&  K(X) \ar[dd] \ar[rr] &&    \ar[dd]  K(A)  \\
\Kco (A_0)  \ar[dr] \ar'[r][rr]  & & \Kco (A ; A_0) \ar[dr]^\alpha  &  \\
&  \Kco (X)  \ar[rr] &&  W.
}
\end{equation}
The back square is the same as the cartesian square~\eqref{squar.carcompcon}. The
front square is cartesian by the same reasoning. The left square is weakly
cartesian by~\cite[Thm.~A]{KST-Weibel}. So the right square is weakly cartesian as well or in other words $\alpha$ is a
weak equivalence of pro-spectra.
\end{proof}

\begin{rmk}
Proposition~\ref{prop.Kcoray} can be used to show that continuous $K$-theory satisfies
Mayer--Vietories for finite affinoid coverings of affinoid spaces as indicated in~\cite{morrow-hist}.
\end{rmk}

\subsection{Bass fundamental theorem}
	\label{sec.Bass.fundamental}

Let $A$ be a Tate ring with a ring of definition $A_0$ and let $\pi\in A_0$ be a
topologically nilpotent unit in $A$. As both $\bar K(A_0)$ and $\Kco (A_0)$
carry a module structure over $K(A_0)$ and $\Kco (A) $ carries a compatible module
structure over $K(A)$ the techniques of Subsection~\ref{sec.bassdel} apply. In particular
we obtain a commutative diagram of fibre sequences
\[
  \xymatrix{
    \bar K(A_0 ) \ar[r] \ar[d]_{\lambda_1} &  \Kco (A_0)\ar[d]_{\lambda_2}  \ar[r] & \Kco ( A;A_0 ) \ar[d]_{\lambda_3} \\
    \Lambda \bar K(A_0 ) \ar[r] & \Lambda \Kco (A_0) \ar[r] & \Lambda\Kco ( A;A_0 )
  }
\]
in which the vertical maps are equivalences of pro-spectra. Indeed, in order to show that
$\lambda_1$ and $\lambda_2$ are equivalences  one can replace in the definition of $\Lambda$ the rings $A_0\langle
t^{\pm} \rangle$ and $A_0 \langle t, t^{-1} \rangle $ by their dense subrings $A_0[
t^{\pm} ]$ and $A_0 [ t, t^{-1} ]$. This is obvious for $\Kco$, and for $\bar K$ it follows from~\cite[Prop.~3.19]{thomason}. So 
$\lambda_1$ and $\lambda_2$ are  equivalences as a consequence of~\cite[Thm.~6.6]{thomason}.
This means the canonical map $\lambda\colon \Kco  \to \Lambda \Kco $ is an equivalence, so we deduce that
the Bass fundamental theorem holds  for $\Kco $, see Proposition~\ref{prop.absbass}.

\begin{prop}\label{prop.comalgcont}  Continuous $K$-theory of a Tate ring $A$ satisfies:
  \begin{itemize}
  \item[(i)] $\lambda \colon \Kco_i(A)\to \Sigma \Kco_{i+1}(A) $ is an isomorphism of pro-groups.
  \item[(ii)] The canonical map $K_0(A)\to \Kco_0(A)$ is an isomorphism of pro-groups.
  \item[(iii)] The pro-groups $\Kco_i(A)$ are constant for $i\le 0$ and agree with the
    groups $K^{\rm top}_i (A) $ defined and studied in~\cite{calvo}.
  \end{itemize}
\end{prop}

\begin{proof}
Part (i) is clear from what is explained above. Part (ii) follows from part (i),
Lemma~\ref{algK.lamdense} and Proposition~\ref{algk.bassfund}. Part (iii) is clear in view
of part (ii) and the Bass fundamental theorem for $\Kco$.
\end{proof}

\begin{rmk}
We have the following exact sequence in the category of $\Ab$ of abelian groups.
\[
0\to {\lim}^{1} \, \Kco_1(A) \to \pi_0 \lim \Kco(A) \to {\lim}^{0} \, \Kco_0(A) \to 0 
\]
The ${\lim}^{1}$-term vanishes by Lemma~\ref{lem.calK1cont}. So using Proposition~\ref{prop.comalgcont}(ii)
we deduce that $K_0(A)\xrightarrow{\simeq} \pi_0 \lim \Kco(A)$ is an isomorphism. Here
$\lim^0$ and $\lim^1$ are the usual limit functors from the category $\Pro (\Ab ) $  to
$\Ab$ and
$\lim\colon\Pro(\Sp) \to \Sp$ is the right adjoint to the inclusion.
\end{rmk}

Using Gabber's rigidity theorem we can compute continuous $K$-theory of Tate rings with finite coefficients.
Let $\ell$ be a positive integer. If $F$ is a functor from adic rings or from Tate rings to spectra or pro-spectra, we write $F(A; \Z/\ell)$ for the (level-wise) smash product of $F(A)$ with the mod-$\ell$ Moore spectrum $\mathbb{S}/\ell$.

Let $A$ be a Tate ring with a ring of definition $A_{0}$, and let $\ell$ be a positive integer which is invertible in $A_{0}$.
\begin{lem}
	\label{lem.Kcont-finite-coeffs}
The canonical map $K(A) \to \Kco(A)$ induces isomorphisms $K_{i}(A; \Z/\ell)
\xrightarrow{\simeq}  \Kco_{i}(A; \Z/\ell)$ and $K_{0}(A)/\ell \xrightarrow{\simeq}
\Kco_{0}(A)/\ell$ for $i\ge 1$.
\end{lem}

\begin{proof}
The statement about $K_{0}$ follows directly from Proposition~\ref{prop.comalgcont}(ii).
Since smashing with $\mathbb{S}/\ell$ is an exact functor, the cartesian  square \eqref{squar.carcompcon}  and Proposition~\ref{prop.welldefinedcont} yield the weakly cartesian square
\[
\xymatrix{
K(A_{0}; \Z/\ell) \ar[r] \ar[d] & K(A; \Z/\ell) \ar[d] \\
\Kco(A_{0}; \Z/\ell) \ar[r] & \Kco(A; \Z/\ell).
}
\]
Since $A_{0}$ is $\pi$-adically complete (where $\pi$ generates the topology on $A_{0}$), Gabber's rigidity theorem \cite[Thm.~1]{Gabber} implies that the left vertical map induces an isomorphism on $\pi_{i}$ for every $i>0$. This implies that  $K_{i}(A;\Z/\ell) \to \Kco_{i}(A; \Z/\ell)$ is an isomorphism for $i>1$ and  injective for $i=1$. In order to show that it is also surjective for $i=1$, we consider the following diagram in which the exact rows come from the Bockstein sequence and $[\ell]$ denotes the $\ell$-torsion subgroup:
\[
\xymatrix{
0 \ar[r] & K_{1}(A)/\ell \ar[r]\ar[d] & K_{1}(A;\Z/\ell) \ar[r] \ar[d] & K_{0}(A)[\ell] \ar[r]\ar[d] & 0    \\
0 \ar[r] & \Kco_{1}(A)/\ell \ar[r] & \Kco_{1}(A;\Z/\ell) \ar[r] & \Kco_{0}(A)[\ell] \ar[r] & 0
}
\]
The right vertical map is an isomorphism by Proposition~\ref{prop.comalgcont}(ii). The left vertical map is surjective by Lemma~\ref{lem.calK1cont}. Hence the middle vertical map is surjective, too.
\end{proof}

\subsection{Analytic pro-homotopy invariance}\label{subsec.anahom}

Let $A$ be an adic ring resp.\ a Tate ring and let $\pi\in A$ be an element generating the
topology on $A$ resp.\ a topologically nilpotent unit.
Let $\mathbf C_A$ be the comma category of adic rings resp.\ Tate rings over $A$. In the
following we fix an integer $m>0$ and we write $A\langle \underline t \rangle = A\langle
t_1, \ldots , t_m \rangle   $.
For a functor $F$ from $\mathbf C_A$ to $\Pro(\Sp)$ we define the new
functor $\hat N^\pi F=\hat N^\pi_m F$ on the same categories by
\[
\hat N^\pi F(A) = \cofib (  F(A) \to \prolim{\underline t\mapsto \pi \underline t} F( A\langle
\underline t \rangle )).
\]
Recall that the polynomial version $N^\pi$ of $\hat N^\pi$ has been studied in Subsection~\ref{subsec.alghomo}.

\smallskip

Let $A$ be a Tate ring, $A_0\subset A$ be a ring of definition and $\pi\in A_0$ an element
which is a topologically nilpotent unit in $A$.

\begin{lem}\label{lem.homintateco}
  We have:
  \begin{itemize}
    \item[(i)] $\hat N^\pi  \Kco(A_0)$ is a contractible pro-spectrum.
   \end{itemize}  
 If moreover $A$ is regular and satisfies condition $(\dagger)_A$ we have:
 \begin{itemize}
  \item[(ii)] $\hat  N^\pi \bar K(A_0)$ is a weakly contractible pro-spectrum.
  \item[(iii)] There is a natural weak equivalence
\[   \hNpi K(A) \simeq 
\cofib \big(  K(\Ao) \to \prolim{t\mapsto \pi t} K(\Ao[[\underline t]]) \big).\]
\item[(iv)]
In the case of one variable, i.e.\ $m=1$,
there is a natural isomorphism of pro-abelian groups
\[ \pi_1(\hNpi K(A)) \cong \prolim{[\pi]} W(\Ao),\]
where $W(\Ao)= (1+ t\Ao[[t]])^\times$ is the ring of big Witt vectors and the right hand side denotes the pro-system
\[ \xymatrix{
W(\Ao)& \ar[l]_{[\pi]} W(\Ao) & \ar[l]_{[\pi]} W(\Ao)& \ar[l]_{[\pi]}\cdots
}\]
with $[\pi]=(1-\pi t)^{-1}\in W(\Ao)$.
\end{itemize}  
\end{lem}

\begin{proof}
Part~(i) is clear as for fixed $j>0$ the map \[A_0/(\pi^j)\to \prolim{\underline t \mapsto
  \pi \underline t
  } A\langle \underline t
\rangle /(\pi^j)\] is an isomorphism of pro-rings.
For (ii) we use that the canonical map $N^\pi \bar K(A_0)\xrightarrow{\simeq}\hat  N^\pi \bar K(A_0)  $ is
an isomorphism by excision~\cite[Thm.~7.1]{thomason}, where the left side is as defined in
Subsection~\ref{subsec.alghomo}. From localization theory~\cite[Thm.~7.4]{thomason} we get a fibre
sequence
\begin{equation}\label{eq.fibseNsup}
N^\pi \bar K(A_0) \to N^\pi  K(A_0) \to N^\pi K(A).
\end{equation}
The right pro-spectrum in~\eqref{eq.fibseNsup} is contractible as $A$ is regular and $K$-theory of regular rings is
$\mathbb A^1$-homotopy invariant~\cite[Thm.~V.6.3]{K-book}. Up to weak equivalence  the pro-spectrum in the 
middle of~\eqref{eq.fibseNsup} is independent of the choice of the ring of definition
$A_0$ by Remark~\ref{rmk.lem.exci}. So we can assume by condition $(\dagger)_A$ that $A_0$ is noetherian and that
there exists a proper morphism $\tilde X\to \Spec (A_0)$ which is an isomorphism over
$\Spec (A)$ and such that $\tilde X$ is regular. Corollary~\ref{khpro.cor} finally implies that
$N^\pi  K(A_0)$ is weakly contractible. Together this implies that also the pro-spectrum on the left of~\eqref{eq.fibseNsup} is weakly contractible proving (ii).
For (iii), note that the natural map
\begin{equation*}
\hNpi K(A_0) \to \hNpi K(A)
\end{equation*}
is a weak equivalence by (ii). From the commutative diagram
\[\xymatrix{
\Ao \langle \underline t \rangle  \ar[r]\ar[d]_{\underline t\to \pi t} & \Ao[[\underline t ]]\ar[d]^{\underline t\mapsto \pi
  \underline t}\ar[dl]^{\underline t\mapsto \pi \underline t} \\
\Ao \langle \underline t \rangle \ar[r] & \Ao[[\underline t]] \\
}\]
we obtain an isomorphism of pro-simplicial rings
\[ \prolim{\underline t\mapsto \pi \underline t} \Ao \langle \underline t \rangle  \
  \simeq \prolim{\underline t\mapsto \pi  \underline t} \Ao[[\underline  t]].\]
This proves (iii). Part (iv) follows from part (iii) and the isomorphism
\[ K_1(\Ao[[t]])/K_1(\Ao ) \simeq (1+ t\Ao[[t]])^\times,\]
see \cite[Lem.\ III.2.4]{K-book}.
\end{proof}

From the fibre sequence
\[
\hat  N^\pi \bar K(A_0) \to \hat  N^\pi  \Kco (A_0) \to \hat  N^\pi \Kco (A)  
\]
and Lemma~\ref{lem.homintateco} we immediately deduce:

\begin{prop}\label{cont.homotopprop}
   For a regular Tate ring $A$ which satisfies condition $(\dagger)_A$ the
   canonical map
   \[
     \Kco(A)\xrightarrow{\simeq} \prolim{\underline t\mapsto \pi \underline t} \Kco(A\langle \underline t \rangle )
   \]
   is a weak equivalence of pro-spectra.
 \end{prop}

 As a corollary to Proposition~\ref{cont.homotopprop} and
 Proposition~\ref{prop.comalgcont}(ii) we recover the following pro-homotopy invariance result for $K_{0}$, which has already
 been shown in~\cite{proHIPic} using a more direct approach.

 \begin{cor}\label{cor.k0homotopyi}
   For a regular Tate ring $A$ which satisfies condition $(\dagger)_A$ the
   canonical map
   \[
     K_0(A)\xrightarrow{\simeq} \prolim{\underline t\mapsto \pi \underline t} K_0(A\langle  \underline t\rangle )
   \]
is an isomorphism of pro-groups.
\end{cor}

For $K_{1}$ we at least have the following weak pro-homotopy invariance result:
 \begin{prop}
If $A$ is regular and satisfies condition $(\dagger)_A$, then
\[R^i\underset{t\to \pi t}{\lim}\; \pi_1\hNpi K(A)=0\quad \hbox{ for } i\geq 0. \]
\end{prop}
\begin{proof}
We have a commutative diagram with exact rows
\[\xymatrix{
0\ar[r] & W(\Ao) \ar[r]^{[\pi]^{n+1}} \ar[d]^{[\pi]} 
& W(\Ao) \ar[r]\ar[d]^{=} & W(\Ao)/[\pi]^{n+1} \ar[r]\ar[d]& 0 \phantom{\ .}\\
0\ar[r] & W(\Ao) \ar[r]^{[\pi]^{n}} & W(\Ao) \ar[r] & W(\Ao)/[\pi]^{n} \ar[r]& 0\ .\\
}
\]
From this we get
\[ R^i\underset{t\to \pi t}{\lim} \; W(\Ao) \cong 
\begin{cases} 
\ker(W(\Ao) \to \widehat{W(\Ao)}_{[\pi]}),&\text{ for } i=0, \\
\coker(W(\Ao) \to \widehat{W(\Ao)}_{[\pi]}),&\text{ for } i=1, \\
0,&\text{ otherwise, }
\end{cases}
\]
where $\widehat{W(\Ao)}_{[\pi]}$ is the $[\pi]$-adic completion of $W(\Ao)$.
Hence the assertion follows from Lemma \ref{lem.homintateco} and the following lemma.
\end{proof}

\begin{lemma}
The ring $W(\Ao)$ is $[\pi]$-adically complete.
\end{lemma}
\begin{proof}
We have 
\[
f \in [\pi]^n\cdot W(\Ao)
\;\;\text{ for }
f=1+\underset{i\geq 1}{\sum} a_i t^i\; \in W(\Ao)
\]
if and only if there are $c_i\in \Ao$, $i\geq 1$, such that
\[ 
a_{i} = \pi^{ni}c_{i} \text{ for } i \geq 1.
\]
The lemma follows from this by direct computations.
\end{proof}


\section{Analytic \textit{K}-theory}
	\label{sec.anK}

\subsection{Basic constructions}
	\label{sec.basic}

Let $A$ be a $\pi$-adic ring for some $\pi \in A$. In order to define analytic $K$-theory
we consider the system of simplicial rings
\begin{equation}\label{eq.an.pros}
  \xymatrix{
\cdots \ar[r]  &  A\langle \Delta_{\pi^3}\rangle \ar[r]^{\Psi_\pi^{\pi^2}}& A\langle \Delta_{\pi^2}\rangle \ar[r]^{\Psi_\pi^{\pi^1}} & A\langle \Delta_{\pi^1}\rangle
\ar[r]^{\Psi_\pi^{\pi^0}} & A\langle \Delta_{1}\rangle  ,
  }
\end{equation}
see Subsection~\ref{sec.adicrin}. We would like to see that the pro-system of simplicial
rings~\eqref{eq.an.pros} does not depend on the choice of $\pi$ up to canonical
isomorphism. Indeed, let $\varpi   \in A$ be another element generating the topology of $A$
and consider the analog of the pro-system \eqref{eq.an.pros} with $\pi$ replaced by
$\varpi$. We have $\varpi^n=a\pi$ and that $\pi^n =b
\varpi$ for some positive integer $n$ and $a,b\in A$.
In view of the commutative square
\[
  \xymatrix{
     A\langle \Delta_{\varpi^{n(j+1)}}\rangle \ar[r]^-{\Psi_{\varpi^n}^{\varpi^{nj}}}
     \ar[d]_{\Psi_{a^{j+1}}^{\pi^{j+1}}}&  A\langle
     \Delta_{\varpi^{nj}}\rangle   \ar[d]^{\Psi_{a^{j}}^{\pi^{j}}}\\
    A\langle \Delta_{\pi^{j+1}}\rangle \ar[r]^-{\Psi_\pi^{\pi^j}} &  A\langle \Delta_{\pi^j}\rangle
  }
\]
and the analogous square with $\pi$ and $\varpi$ interchanged and $a$ replaced by $b$ we
get a canonical isomorphism of pro-simplicial rings $\Phi^\pi_\varpi$ such that the
diagram
\[
  \xymatrix@R=2.5ex@C=1.3ex{
\prolim{j} A\langle \Delta_{\pi^j}\rangle \ar[rr]_\simeq^{\Phi^\pi_\varpi}  \ar[rd] &  & \prolim{j} A\langle
\Delta_{\varpi^j}\rangle \ar[dl] \\
& A\langle \Delta_{1}\rangle  &
  }
\]
commutes.

If $A$ is a Tate ring and $\pi \in A$ is a topologically nilpotent unit, one similarly constructs a pro-simplicial ring 
$\lprolim{j} A\langle \Delta_{\pi^j}\rangle$ and one shows that  it does not
depend on the choice of $\pi$ up to canonical isomorphism.

In the remainder of this subsection we assume that $A$ is either an adic ring or a Tate ring.
Let $F$ be a functor from the category of adic rings or from the category of
Tate rings to connective spectra. We set 
\[
F^{(j)}(A) = F(A\langle \Delta_{\pi^j}\rangle )
\]
and
define a new functor $F^\an$ to the $\infty$-category of
pro-spectra by
\[
F^\an (A) = \prolim{j} F^{(j)}(A).  
\]
Here $F(A\langle \Delta_{\pi^j}\rangle)$ denotes the geometric realization of the simplicial spectrum $[p] \mapsto F(A\langle \Delta^{p}_{\pi^j}\rangle)$.  

\begin{lem}
	\label{lem.an-ss}
For any $A$ and any functor $F$ as above, there is a convergent spectral sequence of pro-abelian groups
    \[
	E^{1}_{pq} = \prolim{j} \pi_{q} F(A \langle \Delta^{p}_{\pi^{j}} \rangle ) \Longrightarrow \pi_{p+q} F^{\an}(A) 
    \]
which yields an isomorphism of pro-groups
   \[
  \pi_0 F^\an(A) = \prolim{j} \coker( \pi_0 F(A \langle \Delta_{\pi^j}^1\rangle)
      \xrightarrow{ \partial_{0} - \partial_{1} } \pi_0 F(A)  ).
    \]
\end{lem}

\begin{proof}
For fixed $j$ we have the classical convergent  $E^{1}$-spectral sequence of a simplicial spectrum (see e.g.~\cite[Prop.~1.2.4.5]{HA})
\[
 \pi_{q} F(A \langle \Delta^{p}_{\pi^{j}} \rangle ) \Longrightarrow \pi_{p+q} F^{(j)}(A). 
\]
Since $F$ has valued in connective spectra, 
these are bounded uniformly in~$j$. They hence yield the desired convergent spectral sequence in the abelian category $\Pro(\Ab)$.
The asserted isomorphism is an immediate consequence.
\end{proof}

\begin{lem}
	\label{lem:simplicial-homotopy}
For fixed $j$, the maps of simplicial rings $A\langle t \rangle \langle \Delta_{\pi^{j}} \rangle  \to  A\langle t \rangle \langle \Delta_{\pi^{j}} \rangle$ induced by $t \mapsto \pi^{j}t$ and $t \mapsto 0$, respectively, are simplicially homotopic.
\end{lem}

\begin{proof}
For fixed $n$ and $0 \leq \ell \leq n$ we have ring homomorphisms 
\[
h_{\ell}\colon   A\langle t \rangle \langle \Delta^{n}_{\pi^{j}} \rangle  \to  A\langle t \rangle \langle \Delta^{n+1}_{\pi^{j}} \rangle
\]
induced by $h_{\ell}(t) = t(t_{0} + \dots + t_{\ell})$ and $h_{\ell}(t_{i}) = s_{\ell}(t_{i})$, where $s_{\ell}$ is the $\ell$-th degeneracy map and the $t_{i}$ are the coordinates on $\Delta^{n}_{\pi^{j}}$. These give the desired homotopy.
\end{proof}

Given variables $\underline t =(t_1,\ldots , t_m)$, we let 
\[
\hat N^\pi_m F(A)  = {\rm cofib} (F(A)  \to  \prolim{\underline t\mapsto \pi \underline t}
  F(A\langle \underline t \rangle ) )
  \]
as in Subsection~\ref{subsec.anahom}.

\begin{prop}
	\label{prop.Fan-prohi}
For each $j$, the natural map   
\[
F^{(j)}(A) \to \prolim{t \mapsto \pi t} F^{(j)}(A \langle t \rangle)
\]
is an equivalence of pro-spectra.
In particular, the 
   pro-spectrum $\hat N^\pi_m F^\an (A) $ is contractible for any $m\ge 0$.
\end{prop}

\begin{proof}
The map $A \langle t \rangle \to A$, $t \mapsto 0$, induces a left inverse. As any functor preserves simplicial homotopies, Lemma~\ref{lem:simplicial-homotopy} implies that each $j$-fold composition of transition maps $F^{(j)}(A \langle t \rangle) \to F^{(j)}(A \langle t \rangle)$ factors through $F^{(j)}(A)$ up to homotopy. This implies the first claim. Taking $\prolim{}$ over $j$ implies that 
\[
F^{\an}(A) \xrightarrow{\simeq} \prolim{t \mapsto \pi t} F^{\an}(A \langle t \rangle)
\]
is also an equivalence. By induction on $m$  this gives the second assertion.
\end{proof}

If $G\to F$ is the connected covering of $F$,  we consider the fibre sequence
\[
G\to F \to  F_0
\]
where
$F_0 = { H}(\pi_0 F)$ denotes the Eilenberg--MacLane spectrum living in homotopy degree $0$.

\begin{lem} 
	\label{lem.connected.cover.an}
For an adic ring resp.\ Tate ring $A$
  there is a fibre sequence of pro-spectra
  \[
G^\an(A) \to F^\an (A) \to \prolim{j}  F_0 (  A\langle \Delta_{\pi^j}\rangle ). 
\]
If $\hat N^\pi_m F_0(A)$ is  contractible for any number of variables
$m\ge 0$,
then there is a weak fibre sequence of pro-spectra
\[
G^\an(A) \to F^\an(A) \to  F_0 (  A ). 
\]
\end{lem}

Note that $\hat N^\pi_m F_0(A)$ is  contractible if and only if the natural map $F_{0}(A) \to \lprolim{j} F_{0}(A \langle \Delta^{m}_{\pi^{j}} \rangle)$ is an equivalence.

\begin{proof}
For every $j$ and $m$ there is a fibre sequence of spectra
\[
G(A\langle \Delta^{m}_{\pi^{j}} \rangle )  \to F( A\langle \Delta^{m}_{\pi^{j}} \rangle ) \to F_{0}( A\langle \Delta^{m}_{\pi^{j}} \rangle ).
\]
Since fibre sequences of spectra are preserved by geometric realization, this induces a fibre sequence
\[
G^{(j)}(A) \to F^{(j)}(A) \to F_{0}(A\langle \Delta_{\pi^{j}} \rangle )
\]
for every $j$. The first claim follows by taking $\lprolim{}$ over $j$.

For the second claim we use the fact that the homotopy pro-groups of the pro-spectrum $\lprolim{j} F_{0}(A \langle \Delta_{\pi^{j}} \rangle )$ are isomorphic to the homotopy pro-groups of the pro-simplicial abelian group $\lprolim{j} (\pi_{0}F)(A \langle \Delta_{\pi^{j}} \rangle)$.
\end{proof}

\subsection{Analytic \textit{K}-theory of adic rings}
	\label{sec.adicank}

We apply the construction of the previous subsection to the connective $K$-theory of adic rings. Let $A$ be an adic ring, and let $\pi \in A$ generate the topology of $A$.

\begin{defn}
The \emph{analytic $K$-theory of $A$} is defined as the pro-spectrum
\[
\ka (A) = \prolim{j} k(A\langle \Delta_{\pi^j}\rangle ).
\]
\end{defn}
By the discussion in Subsection~\ref{sec.basic} this is independent of the choice of $\pi$ up to canonical equivalence.
In the following, we also use the connective covering
\[
\kco(A) = \prolim{n} k(A/(\pi^{n}))
\]
of the continuous $K$-theory  $\Kco(A)$ of $A$. 
As a bridge between continuous and analytic $K$-theory we further define the `mixed' version
\[
k^{\cont,\an} = \prolim{j,n} k (A\langle  \Delta_{\pi^j}  \rangle /(\pi^n) ) .
\]

\begin{lem}\label{kan.adic.lem1}
  The natural map
  \[
   \kco (A) \xrightarrow{\simeq}    k^{\cont,\an}(A)  
 \]
is an equivalence of pro-spectra.
\end{lem}
\begin{proof}
For fixed $n$ the map $A/(\pi^{n}) \to \lprolim{j} A \langle \Delta_{\pi^{j}} \rangle / (\pi^{n})$ is an isomorphism of pro-simplicial rings. 
Hence it induces an equivalence of pro-spectra 
\[
k(A/(\pi^{n})) \xrightarrow{\simeq} \prolim{j} k( A \langle \Delta_{\pi^{j}} \rangle / (\pi^{n}) )
\]
The claim follows by taking the limit over $n$.
\end{proof}

\begin{prop}\label{kan.adic.prop2}
  The map
  \[
    \ka (A) \to   k^{\cont,\an}(A)  
  \]
 is a weak equivalence of pro-spectra.
\end{prop}

\begin{proof}  
We first reduce to connected $K$-theory. For fixed $j$, $n$, and $m$ we have, as in the
discussion preceding Lemma~\ref{lem.connected.cover.an}, a map of fibre sequences of spectra
\[
\xymatrix@R-1.6ex{
\KGL( A\langle \Delta^{m}_{\pi^{j}} \rangle)  \ar[r] \ar[d] & k( A\langle \Delta^{m}_{\pi^{j}} \rangle ) \ar[d] \ar[r] & K_{0}(A\langle \Delta^{m}_{\pi^{j}} \rangle)   \ar[d]^{\cong}  \\
\KGL( A\langle \Delta^{m}_{\pi^{j}} \rangle / (\pi^{n}) )  \ar[r] & k( A\langle \Delta^{m}_{\pi^{j}} \rangle / (\pi^{n}) )  \ar[r] & K_{0}(A\langle \Delta^{m}_{\pi^{j}} \rangle / (\pi^{n})).
}
\]
The right vertical map is an isomorphism because $A\langle \Delta^{m}_{\pi^{j}} \rangle$
is $\pi$-adically complete, see~\cite[Lemma II.2.2]{K-book}.

It is now enough to show that for each $j$ the map 
\[
\KGL( A\langle \Delta_{\pi^{j}} \rangle ) \to \prolim{n} \KGL( A\langle \Delta_{\pi^{j}} \rangle / (\pi^{n}) )
\]
is a weak equivalence of pro-spectra. We check this on the associated infinite loop spaces, see Lemma~\ref{lemma:weak-equivs-infinite-loop-space}. 
By Lemma~\ref{lem:plus-simplicial} this means we have to show that
the bottom horizontal map in the commutative diagram
\[
\xymatrix@R-1.6ex{
\BGL( A\langle \Delta_{\pi^{j}} \rangle ) \ar[d] \ar[r]^-{\phi}   & \prolim{n} \BGL( A\langle \Delta_{\pi^{j}} \rangle / (\pi^{n}) )  \ar[d]  \\
\BGL( A\langle \Delta_{\pi^{j}} \rangle )^{+} \ar[r]^-{\phi^{+}}   & \prolim{n} \BGL( A\langle \Delta_{\pi^{j}} \rangle / (\pi^{n}) )^{+} 
}
\]
is a weak equivalence in $\Pro(\Spaces_{*}^{0})$. Since the map to the plus-construction is acyclic, the vertical maps induce isomorphisms on integral homology.
\begin{claim}
	\label{claim.phiwe}
The map $\phi$ is a weak equivalence in $\Pro(\Spaces_{*}^{0})$.
\end{claim}
By Lemma~\ref{lem:we-induces-homology-iso} this claim implies that $\phi^{+}$  induces an isomorphism on integral homology. 
Since $\phi^{+}$ is a morphism between countable pro-systems of  connected H-spaces, Proposition~\ref{prop:pro-Whitehead} implies that $\phi^{+}$ is a weak equivalence.
\end{proof}
\begin{proof}[Proof of the Claim]
We have to show that $\phi$ induces an isomorphism on homotopy pro-groups. For this we use that
\[
\pi_{i}(\BGL(A\langle \Delta_{\pi^{j}} \rangle )) \cong \pi_{i-1}(\GL( A\langle \Delta_{\pi^{j}} \rangle )) \quad \text{ for }i\geq 1.
\]
For each $n$, the natural map
\[
\GL(A\langle \Delta_{\pi^{j}} \rangle) \to \GL(A\langle \Delta_{\pi^{j}} \rangle/(\pi^{n}))
\]
is a surjective homomorphism of simplicial groups and hence a Kan fibration. Its fibre is given by the simplicial group
of finite matrices of arbitrary size
\[
1+\pi^{n} \Mat(A\langle \Delta_{\pi^{j}} \rangle)
\]
and it suffices to show that the inclusion 
\[
1+\pi^{n+j} \Mat(A\langle \Delta_{\pi^{j}} \rangle) \hookrightarrow 1+\pi^{n} \Mat(A\langle \Delta_{\pi^{j}} \rangle)
\]
is null-homotopic. We let $\Delta[1]$ be the simplicial 1-simplex, and we 
consider  $t_{1}\in A\langle \Delta^{1}_{\pi^{j}} \rangle $  as a map of simplicial sets
\[
t_{1}\colon \Delta[1] \to A\langle \Delta_{\pi^{j}} \rangle
\]
satisfying $t_{1}(0) = 0$, $t_{1}(1) = \pi^{j}$.
A contracting homotopy is then given by
\begin{align*}
\left(1+\pi^{n+j} \Mat( A\langle \Delta_{\pi^{j}} \rangle) \right) \times \Delta[1] &\to  1+\pi^{n} \Mat(A\langle \Delta_{\pi^{j}} \rangle), \\
(1+a, s) &\mapsto 1 + t(s)\cdot (\pi^{-j}a).  \qedhere
\end{align*}
\end{proof}

Since by Lemma~\ref{kan.adic.lem1} the map of pro-spectra $\kco (A) \xrightarrow{\simeq}  k^{\cont,\an}(A)$ is an equivalence, it has an essentially unique inverse. Composing this with the weak equivalence $\ka (A) \to  k^{\cont,\an}(A) $ of Proposition~\ref{kan.adic.prop2} gives an essentially unique weak equivalence of pro-spectra $\ka(A) \to \kco(A)$ which makes the diagram 
\[
  \xymatrix@R=2.5ex@C=1.3ex{
  \ka (A)   \ar[rr]^{\simeq} \ar[rd]^{\simeq} & &     \kco (A) \ar[ld]_{\simeq} \\
    &     k^{\cont,\an}(A)   & 
  }
\]
commutative. Summarizing we have:

\begin{thm}
	\label{thm:kcontan-adic}
For any $\pi$-adic ring $A$, there is a natural weak equivalence of pro-spectra
\[
\ka (A) \xrightarrow{\simeq} \kco (A).
\]
\end{thm}

\begin{rmk}
In fact, this result holds more generally.
We needed the completeness of a $\pi$-adic ring only in the proof of
Proposition~\ref{kan.adic.prop2}. More precisely, we used that for a $\pi$-adically
complete ring $A$ the map $K_{0}(A) \to K_{0}( A/(\pi) )$ is an isomorphism and $\GL(A)
\to \GL( A/(\pi) )$ is surjective with kernel $1 + \pi \mathrm{M}(A)$. But these two facts
also hold  if $A$ is only assumed to be henselian with respect to $(\pi)$. Hence we get the same result if we work with 
henselian rings and replace the simplicial ring $A\langle \Delta_{\pi^{j}} \rangle$ in the construction of analytic $K$-theory with the  henselisation of~$A[ \Delta_{\pi^{j}} ]$ along the ideal $(\pi)$.
\end{rmk}

\subsection{Analytic \textit{K}-theory of Tate rings}

Let now $A$ be a Tate ring and $\pi \in A$ a topologically nilpotent unit.

\begin{defn}
The \emph{connective analytic $K$-theory of $A$} is defined as the pro-spectrum
\[
\ka(A) = \prolim{j} k(A\langle \Delta_{\pi^j}\rangle ).
\]
\end{defn}

Again, this is well-defined up to canonical equivalence.
As in the previous subsection we will compare analytic $K$-theory and continuous $K$-theory. For this we  need some auxiliary constructions. If $A_{0} \subseteq A$ is a ring of definition with $\pi \in A_{0}$, we set
\begin{align*}
\bar k^\an (A_{0}) &= \prolim{j} \bark (A_{0}\langle \Delta_{\pi^j}\rangle ),  \\
\tildek(A; A_{0}) &= \cofib (\bark(A_0) \to k(A_0)),  \\
\tilde k^\cont (A; A_0) &= \cofib (\bark(A_0) \to \kco (A_0)),  \\
\tilde k^\an (A; A_0) &=  \cofib (\bar k^\an (A_0) \to \ka (A_0 )  ),  \\
\tilde k^{\cont,\an}(A; A_{0}) &= \cofib ( \bar k^{\an}(A_{0}) \to k^{\cont,\an}(A_{0}) )
\end{align*}
(see \eqref{eq.def.Kbar} for the definition of  $\bark$). 
Note that the above spectra are all connective
and that $\tau_{\geq 1} \tildek(A; A_{0}) \simeq \tau_{\geq 1} k(A) \simeq \KGL(A)$.
There is a natural map $\tilde k^{\cont}(A; A_{0}) \to \Kco(A;A_{0})$. It is immediate from the constructions that this map induces an equivalence on connected covers $\tau_{\geq 1} \tilde k^{\cont}(A; A_{0})  \to  \tau_{\geq 1} \Kco(A;A_{0})$. In view of Proposition~\ref{prop.welldefinedcont} we have a natural weak equivalence
\[
\tau_{\geq 1} \tilde k^{\cont}(A; A_{0})  \xrightarrow{\simeq}  \tau_{\geq 1} \Kco(A).
\]

Recall condition $(\dag)_{A}$ from Subsection~\ref{subsec:nonarch.tate}.
\begin{lem}
	\label{lem:k-tilde-an-cont}
Assume that $A$ is regular and satisfies condition $(\dag)_{A}$. Then there is a natural zig-zag of weak equivalences of pro-spectra
  \[
\tilde k^\cont (A; A_0) \xrightarrow{\simeq}  \tilde k^{\cont,\an}(A; A_{0}) \xleftarrow{\simeq}        \tilde k^\an (A; A_0).
\]
\end{lem}

\begin{proof}
Lemma~\ref{lem.homintateco}(ii) and the convergent spectral sequence from Lemma~\ref{lem.an-ss}
\[
E^{1}_{pq} = \prolim{j} \bark_{q}( A_{0} \langle \Delta^{p}_{\pi^{j}} \rangle ) \Longrightarrow \bar k^{\an}_{p+q}(A_{0})
\]
imply that the natural map $\bark(A_{0}) \to \bar k^{\an}(A_{0})$ is a weak equivalence. Combining this with Lemma~\ref{kan.adic.lem1} 
and the definitions of the tilde versions yields the left  weak equivalence of the assertion.
The right-hand weak equivalence follows from Proposition~\ref{kan.adic.prop2}.
\end{proof}

\begin{lem} 
	\label{lem:K0-K0an}
Let $A$, $A_{0}$ be as above.
\begin{itemize}
\item[(i)] There is a natural isomorphism $\tildek_{0}(A; A_{0}) \cong \tilde k^{\an}_{0}(A; A_{0})$.
\item[(ii)] If $A$ is regular and satisfies condition $(\dag)_{A}$, then we have a natural isomorphism
\[
K_{0}(A) \cong \ka_{0}(A)
\]
and 
a natural weak equivalence
\[
\tau_{\geq 1}\tilde k^{\an}(A; A_{0}) \simeq \tau_{\geq 1} \ka(A).
\]

\end{itemize}
\end{lem}

\begin{proof}
Recall that  $k$ and $\tildek$ have common connected cover $\KGL$. 
 For fixed $m$ we have a commutative diagram of pro-groups
\[
\xymatrix@R-1.6ex{
K_{0}(A_{0}) \ar@{->>}[d] \ar[r]^-{\cong} & \prolim{j} K_{0}( A_{0} \langle \Delta^{m}_{\pi^{j}} \rangle )  \ar@{->>}[d]  \\
\tildek_{0}(A; A_{0}) \ar[r] & \prolim{j} \tildek_{0}(A \langle \Delta^{m}_{\pi^{j}} \rangle, A_{0} \langle \Delta^{m}_{\pi^{j}} \rangle).
}
\]
The top horizontal map is an isomorphism by Lemmas~\ref{lem.KcoA0low} and \ref{lem.homintateco}(i). The vertical maps are epimorphisms since they are level-wise surjective by definition of $\tildek$. Hence the bottom horizontal map is an epimorphism, too. Since it has a left inverse, it is an isomorphism.
Now Lemma~\ref{lem.connected.cover.an} implies that we have a weak fibre sequence
\begin{equation}
	\label{kdjbkj}
K^{\GL,\an}(A)  \to \tilde k^{\an}(A; A_{0}) \to \tildek_{0}(A; A_{0}).
\end{equation}
Since $K^{\GL,\an}(A)$ is connected, this gives (i).

Now assume that $A$ satisfies $(\dag)_{A}$.
By Corollary~\ref{cor.k0homotopyi} the map $K_{0}(A) \to \lprolim{j}  K_{0}(A \langle \Delta^{m}_{\pi^{j}} \rangle)$ is an isomorphism for each $m$.
By Lemma~\ref{lem.connected.cover.an} again, we have a weak fibre sequence
\begin{equation}
	\label{sdkfhvksdhjv}
K^{\GL,\an}(A)  \to  k^{\an}(A) \to  K_{0}(A).
\end{equation}
This gives the first statement of (ii).

There is a natural map from \eqref{kdjbkj} to \eqref{sdkfhvksdhjv}. By what we have shown so far, it induces a weak equivalence 
$\tau_{\geq 1}\tilde k^{\an}(A; A_{0}) \xrightarrow{\simeq} \tau_{\geq 1} \ka(A)$.
\end{proof}

The continuous $K$-theory $\Kco(A)$ of $A$ can be recovered from $\tilde k^{\cont}(A; A_{0})$ via the analytic Bass construction:
\begin{lem}
	\label{lem:Bass-construction-cont}
There is a natural  equivalence of pro-spectra
\[
(\tilde k^{\cont})^{B}(A; A_{0}) \xrightarrow{\simeq} \Kco(A; A_{0}).
\]
\end{lem}
\begin{proof}
Since $\Kco(A;A_{0})$ (viewed as functor in $A_{0}$) satisfies the Bass fundamental theorem (see first paragraph of Subsection~\ref{sec.Bass.fundamental}), the natural map 
$\Kco(A; A_{0}) \to (\Kco)^{B}(A;A_{0})$ is an equivalence and we have an induced map $(\tilde k^{\cont})^{B}(A; A_{0}) \to \Kco(A;A_{0})$.
The cofibre $\cofib( \tilde k^{\cont}(A; A_{0}) \to K^{\cont}(A;A_{0}))$ is equivalent to 
$\cofib( \tau_{\leq -1} \bar K(A_{0}) \to \tau_{\leq -1} \Kco(A_{0}) )$. By nil-invariance of negative algebraic $K$-theory of rings, the latter is a constant pro-spectrum, which is moreover 0-truncated. Now Lemma~\ref{lem:Bannihilates} implies that the analytic Bass construction applied to the cofibre (again viewed as functor in $A_{0}$) vanishes, whence $(\tilde k^{\cont})^{B}(A; A_{0}) \to \Kco(A;A_{0})$ is an equivalence.
\end{proof}

\begin{defn} 
The \emph{non-connective analytic $K$-theory} of the Tate ring $A$ is defined  by  the analytic Bass construction applied  to $K^{\GL,\an}$:
\[
\Ka(A) = (K^{\GL,\an})^{B}(A).
\]
\end{defn}

\begin{rmk}\label{rmk.deloop}
Another candidate for the non-connective analytic $K$-theory would be the pro-spectrum $(\ka)^{B}(A)$. If $A$ is regular and satisfies condition $(\dag)_{A}$, both constructions yield weakly equivalent pro-spectra by Lemma~\ref{lem:ktilde-kab-Bass} below. In general, it seems however that $\Ka(A)$ has better  formal properties. For example, one can prove that it admits descent for admissible coverings, as we will explain in a sequel to this note.
\end{rmk}

\begin{prop}
Let $A$ be a Tate ring, and let $I \subset A$ be a closed nilpotent ideal. Then the natural map
\[
\Ka(A) \to \Ka(A/I)
\]
is a weak equivalence.
\end{prop}
\begin{proof}
We will prove in Section~\ref{sec.kv} that $K^{\GL,\an}(A) \to K^{\GL,\an}(A/I)$ is a weak equivalence, see Corollary~\ref{cor:KV-nilinvariant}. This implies the claim by applying the analytic Bass construction~\eqref{def;EB}.
\end{proof}

Choose a ring of definition $A_{0} \subset A$.
\begin{samepage}
\begin{lem}
	\label{lem:ktilde-kab-Bass}\mbox{}
\begin{itemize}
\item[(i)] The natural map 	
	\[
	\Ka(A) = (K^{\GL,\an})^{B}(A) \to (\tilde k^{\an})^{B} (A; A_{0})
	\]
	is a weak equivalence.
\item[(ii)] If $A$ is regular and satisfies $(\dag)_{A}$, then the natural map 
\[
(\tilde k^{\an})^{B} (A; A_{0}) \to (k^{\an})^{B} (A)
\]
is an equivalence.
\end{itemize}
\end{lem}
\end{samepage}

\begin{proof}
Apply the analytic Bass construction to the weak fibre sequence~\eqref{kdjbkj}. Then  Lemma~\ref{lem:Bannihilates} implies~(i).
In the situation of (ii), Lemma~\ref{lem:K0-K0an}(ii) implies that the cofibre of $\tilde k^{\an}(A; A_{0}) \to k^{\an}(A)$ is a constant pro-spectrum concentrated in homotopy degree $0$.  Lemma~\ref{lem:Bannihilates} again implies the assertion.
\end{proof}

\begin{thm}
	\label{thm:Kcont-Kan}
If the Tate ring $A$ is regular and satisfies $(\dag)_{A}$, there is a natural zig-zag of weak equivalences
\[
\Kco(A) \xrightarrow{\simeq} K^{\cont,\an}(A)  \xleftarrow{\simeq} \Ka(A).
\]
\end{thm}
Here $K^{\cont,\an}(A)$ is defined by first applying the analytic Bass construction \eqref{def;EB} to $\tilde k^{\cont,\an}(A;A_{0})$ and then taking the colimit over the poset of rings of definition $A_{0} \subset A$ as in Definition~\ref{def.cont.K}.
\begin{proof}
Fix a ring of definition $A_{0} \subset A$.
We apply the analytic Bass construction  to the zig-zag of Lemma~\ref{lem:k-tilde-an-cont}.
By the previous Lemma $\Ka(A) \to (\tilde  k^{\an})^{B}(A; A_{0})$ is a weak equivalence of pro-spectra. Combining this with Lemma~\ref{lem:Bass-construction-cont} we get the natural  zig-zag of weak equivalences
\[
K^{\cont}(A;A_{0}) \xrightarrow{\simeq} ( \tilde k^{\cont,\an} )^{B} (A; A_{0}) \xleftarrow{\simeq} K^{\an}(A).
\]
Taking the colimit over the rings of definition $A_{0} \subset A$ gives the result.
\end{proof}

\begin{cor}
	\label{cor.kanKan}
Assume that  $A$ is regular and satisfies $(\dag)_{A}$. Then the pro-spectra
$k^{\an}(A)$ and $\tau_{\geq 0} \Ka(A)$ are weakly equivalent.
 In particular, we have an isomorphism of pro-groups
\[
K_{0}(A) \cong \Ka_{0}(A).
\]
\end{cor}

\begin{proof}
By Lemma~\ref{lem:ktilde-kab-Bass} the canonical map $\Ka(A) \to (\ka)^{B}(A)$ is a weak equivalence and we claim that $\ka(A) \to (\ka)^{B}(A)$ induces a weak equivalence on connective covers.
We first consider  connected covers. For this we choose a ring of definition $A_{0} \subset A$. By Theorem~\ref{thm:Kcont-Kan} the pro-spectra $\tau_{\geq 1} \Ka(A)$ and $\tau_{\geq 1} \Kco(A)$ are weakly equivalent. The latter is clearly equivalent to $\tau_{\geq 1} \tilde{k}^{\cont}(A;A_{0})$, which  is weakly equivalent to $\tau_{\geq 1} \tilde{k}^{\an}(A;A_{0})$ by Lemma~\ref{lem:k-tilde-an-cont}. Lemma~\ref{lem:K0-K0an} finally gives the desired weak equivalence with $\tau_{\geq 1} \ka(A)$. One checks that all the weak equivalences involved are compatible with the canonical  maps $\Ka(A) \to (\ka)^{B}(A) \leftarrow \ka(A)$. This gives the claim on connected covers.
For the $\pi_{0}$ statement we look at the composition $K_{0}(A) \to \ka_{0}(A) \to (\ka)^{B}_{0}(A) \cong K_{0}^{\cont}(A)$. The first map is an isomorphism by Lemma~\ref{lem:K0-K0an}, and the composition is an isomorphism by Proposition~\ref{prop.comalgcont}. Hence the second map is an isomorphism, as desired.
\end{proof}

\begin{rmk} 
Let $A$ be a Tate ring with ring of definition $A_{0}$ and $\ell$ a positive integer which is invertible in $A_{0}$. A theorem of Weibel~\cite[Thm.~9.5]{thomason} says that $K$-theory of $\Z[1/\ell]$-algebras with $\Z/\ell$-coefficients is homotopy invariant.
Using this one can show that the analogs of Lemmas~\ref{lem:k-tilde-an-cont}, \ref{lem:K0-K0an}, and \ref{lem:ktilde-kab-Bass} with $\Z/\ell$-coefficients are true without assuming that $A$ is regular. Hence the pro-spectra $\Kco(A;\Z/\ell)$ and $\Ka(A; \Z/\ell)$ are weakly equivalent. Combining this with Lemma~\ref{lem.Kcont-finite-coeffs}, we also get isomorphisms $K_{i}(A;\Z/\ell) \cong \Ka_{i}(A;\Z/\ell)$ for $i >0$.
\end{rmk}


\section{Analytic \textit{KV}-theory}\label{sec.kv}

In this section we briefly explain the main points of the $KV$-analog of our analytic $K$-theory.
Here $A = (A, \norm)$ is a complete normed ring as in Subsection~\ref{sec:normed-rings}. 
In particular, $A$ need  not be  commutative nor unital.
If  $A$ is not unital, we set $\GL(A):= \ker(\GL(\widetilde{A}) \to \GL(\Z))$ using the unitalization~$\widetilde A$.
As in Subsection~\ref{sec.basic} we get a pro-simplicial  ring
\[
\prolim{\rho} A\langle \Delta_{\rho} \rangle.
\]

\begin{defn} 
	\label{def-KVan}
The \emph{analytic $KV$-theory} of $A$ is defined as the pro-simplicial set
\[
KV^{\an}(A) = \prolim{\rho} \BGL( A\langle \Delta_{\rho} \rangle ).
\]
Its $i$-th homotopy pro-group is denoted by $KV^{\an}_{i}(A)$.
\end{defn}

We may also view $KV^{\an}(A)$ as a pro-space, i.e.~as an object of the $\infty$-category $\Pro(\Spaces_{*}^{0})$.
Directly from the definition we get:

\begin{lem}
	\label{lem.KVspace}
There are natural  isomorphisms
\[
KV^{\an}_{i}(A) \cong  \prolim{\rho} \pi_{i-1}\GL ( A\langle \Delta_{\rho} \rangle ).
\]
\end{lem}

If $a = (a_{ij})$ is a matrix with coefficients in $A$, we set $\|a\| = \max\{ \|a_{ij}\| \}$. This defines a norm on the ring of matrices $\mathrm{M}(A)$.

\begin{lem}
	\label{lem:KV1-explicit}
There is a natural isomorphism 
\[
KV^{\an}_{1}(A) \cong \prolim{\rho} \GL(A) / \GL(A)_{\rho}.
\]
Here $\GL(A)_{\rho}$ is the  subgroup of $\GL(A)$ generated by matrices $g$ such that $\lim_{n\to\infty} \|(g-1)^{n}\| \rho^{n} = 0$, which is normal.
\end{lem}

\begin{proof}
By the previous  Lemma~\ref{lem.KVspace}, it suffices to show that $\pi_{0}\GL(A \langle \Delta_{\rho} \rangle) \cong \GL(A) / \GL(A)_{\rho}$.
We can identify $\pi_{0}\GL(A \langle \Delta_{\rho} \rangle) = \GL(A)/G$ where
 $G$ is the subgroup consisting of matrices $g$ such that there exists a matrix $g(t) \in \GL(A \langle t \rangle_{\rho})$ with $g(0) = 1$ and $g(1) = g$. The same proof as in \cite[App.~3]{Karoubi-Villamayor} shows that $G = \GL(A)_{\rho}$.
\end{proof}

\begin{rmk}  \label{rem-KVanKcont2}
 An $s$-bounded homomorphism between complete normed rings $\phi\colon A \to B$ (see Definition~\ref{def.s-bounded}) induces a map of simplicial rings $A\langle \Delta_{\rho^{1/s}}\rangle \to B\langle \Delta_{\rho}\rangle$ and hence a map of pro-spaces $KV^{\an}(A) \to KV^{\an}(B)$.
This implies that, up to natural equivalence, $KV^{\an}(A)$ depends only on the equivalence class of the norm. Thus by Example \ref{ex:normed-rings}(iii), the analytic $KV$-theory of a  Tate ring is well defined.
\end{rmk}

\begin{lem}
	\label{lem:KV-Kan-connective-cover}
If $A$ is a regular Tate ring satisfying $(\dag)_{A}$, then we have natural weak equivalences 
\[
KV^{\an}(A) \simeq \Omega^{\infty} \tau_{\geq 1} k^{\an}(A) \simeq \Omega^{\infty}\tau_{\geq 1} \Ka(A).
\]
\end{lem}
\begin{proof}
Since for each $\rho \geq 1$ the simplicial ring $A\langle \Delta_{\rho} \rangle$ is connected, Lemma~\ref{lem.KGLinfiniteloops} 
implies that $KV^{\an}(A)$ is naturally  equivalent to $\lprolim{\rho} \Omega^{\infty} \KGL( A\langle \Delta_{\rho} \rangle )$.
By Remark~\ref{rmk.Tatenormedsimplicial} the latter is equivalent to $\Omega^{\infty}K^{\GL,\an}(A)$. 
Hence the claim follows  from \eqref{sdkfhvksdhjv} and Corollary~\ref{cor.kanKan}.
\end{proof}

The same proof as that of Proposition~\ref{prop.Fan-prohi} shows:
\begin{prop}
	\label{prop:pro-homotopy-invariance}
The inclusions $A \subset A\langle t\rangle_{\sigma}$ induce an equivalence of pro-spaces
\[
KV^{\an}(A) \xrightarrow{\sim} \prolim{\sigma} KV^{\an}(A\langle t\rangle_{\sigma}).
\]
\end{prop}

We write $\GL(A\langle t_{1}, \dots, t_{n}\rangle_{\rho})_{0}$ for the subgroup consisting of matrices $f$ such that $f(0,\dots, 0)=1$.
\begin{defn}
A bounded  homomorphism $\phi\colon A \to B$ of complete normed rings  is called a \emph{pro-$\GL$-fibration} if for every $n$ the induced map
\[
\prolim{\rho} \GL(A\langle t_{1}, \dots, t_{n}\rangle_{\rho})_{0} \to  \prolim{\rho} \GL(B\langle t_{1}, \dots, t_{n}\rangle_{\rho})_{0}
\]
is an epimorphism of pro-sets. 
\end{defn}

\begin{rmks}
\phantomsection 
\label{rmk:pro-GL-fibrations}
\begin{itemize}
\item[(i)] A pro-$\GL$-fibration is surjective. (For $b\in B$ consider $\left(\begin{smallmatrix} 1 & t_{1}b \\ 0 & 1\end{smallmatrix}\right)$.)

\item[(ii)] If $I\subset A$ is a closed ideal, we equip $A/I$ with the residue norm. If $I$ is moreover nilpotent, then $A \to A/I$ is a pro-$\GL$-fibration.
(Each $I\langle \Delta^{n}_{\rho}\rangle$ is nilpotent, hence $\GL(A\langle \Delta^{n}_{\rho}\rangle) \to \GL(A/I\langle \Delta^{n}_{\rho}\rangle)$ is surjective.)

\end{itemize}
\end{rmks}

Let $X = \lprolim{j} X^{(j)}$ be a pro-simplicial set. Write $X_{n}= \lprolim{j} X^{(j)}_{n}$ for the pro-set of its $n$-simplices.
We call $X$ \emph{weakly discrete} if, for every $n$, the unique degeneracy map
$s:=s_{0}^{n}\colon X_{0} \to X_{n}$ is an isomorphism of pro-sets.
\begin{lemma}\label{lem:weakly-discrete}
If $X\to Y$ is a map of pro-simplicial sets such that each $X_{n}\to Y_{n}$ is an epimorphism of pro-sets, and if $X$ is weakly discrete, then so is $Y$.
\end{lemma}
\begin{proof}
Consider the commutative diagram of pro-sets
\[
\xymatrix{
X_{0} \ar[r]^{s}_{\cong} \ar[d] & X_{n} \ar@{->>}[d] \\
Y_{0} \ar[r]^{s} & Y_{n}.
}
\]
Since $X_{0} \to X_{n}$ is an isomorphism and $X_{n} \to Y_{n}$ is an epimorphism, $s\colon Y_{0} \to Y_{n}$ is also an epimorphism. Since any composition of face maps $Y_{n} \to Y_{0}$ is a left-inverse, $s$ is indeed an isomorphism.
\end{proof}
\begin{lemma}\label{lem:homotopy-groups-weakly-discrete}
If $X$ is a pointed weakly discrete pro-simplicial set which is level-wise Kan, then $\pi_{0}(X)\cong X_{0}$ and the higher homotopy pro-groups vanish.
\end{lemma}

\begin{proof}
The analog of Lemma~\ref{lem:homotopy-groups-weakly-discrete} for truncated pro-simplicial
sets is clear as they are weakly discrete if and only if they are equivalent to
pro-systems of discrete simplicial sets. As the map $X\to X^\natural$ to the Postnikov
tower, see
\cite[Ch.\ 4]{AM}, is a weak equivalence under our assumption on $X$, one immediately reduces
the lemma to the truncated case.
\end{proof}

Analytic $KV$-theory satisfies excision for pro-GL-fibrations:
\begin{prop}\label{prop:excision} 
If $A\to B$ is a pro-$\GL$-fibration of complete normed rings with kernel $I$, there is a long exact sequence of pro-abelian groups
\[
\dots \to KV^{\an}_{i}(I) \to KV^{\an}_{i}(A) \to KV^{\an}_{i}(B) \to KV^{\an}_{i-1}(I) \to \dots
\]
ending in $KV^{\an}_{1}(B) \to K_{0}(I) \to K_{0}(A) \to K_{0}(B)$.
\end{prop}

\begin{proof}
Let $G_{\rho}$ denote the image of the simplicial group $\GL(A\langle \Delta_{\rho}\rangle)$ in $\GL(B\langle \Delta_{\rho}\rangle)$.
For each $\rho$ we have a short exact sequence of simplicial groups
\begin{equation}\label{seq:GL}
1 \to \GL(I\langle \Delta_{\rho}\rangle) \to \GL(A\langle \Delta_{\rho}\rangle) \to G_{\rho} \to 1.
\end{equation}
Since $A\to B$ is a pro-$\GL$-fibration, 
\[
\prolim{\rho} \GL(A\langle \Delta^{n}_{\rho}\rangle) \times \GL(B) \to \prolim{\rho} \GL(B\langle \Delta^{n}_{\rho}\rangle)
\]
is  an epimorphism of pro-sets.
This implies that 
\[
\GL(B) \to \prolim{\rho} \GL(B\langle \Delta^{n}_{\rho}\rangle)/G_{\rho,n}
\]
is an epimorphism, too. Hence the pro-simplicial set $\lprolim{\rho} \GL(B\langle \Delta_{\rho}\rangle)/G_{\rho}$ is weakly discrete by Lemma \ref{lem:weakly-discrete}.
For each $\rho$, the projection 
\[
\GL(B\langle \Delta_{\rho}\rangle) \to \GL(B\langle \Delta_{\rho}\rangle)/G_{\rho}
\]
is a Kan fibration between Kan complexes with fibre $G_{\rho}$. Together with Lemma \ref{lem:homotopy-groups-weakly-discrete} this implies that
\[
\prolim{\rho} \pi_{n}(G_{\rho}) \cong \prolim{\rho} \pi_{n}(\GL(B\langle \Delta_{\rho}\rangle)) = KV^{\an}_{n+1}(B)
\]
for $n>0$ and that we have an exact sequence
\begin{multline}\label{seq:GL-0}
0 \to \prolim{\rho} \pi_{0}(G_{\rho})  \to \prolim{\rho} \pi_{0}(\GL(B\langle \Delta_{\rho}\rangle)) \to \\ \to  \prolim{\rho} \GL(B\langle \Delta^{0}_{\rho}\rangle)/G_{\rho,0} \to 0.
\end{multline}
Note that $G_{\rho,0} = \im(\GL(A) \to \GL(B))$ is a normal subgroup of $\GL(B\langle \Delta^{0}_{\rho}\rangle) = \GL(B)$
as $A \to B$ is surjective by Remark~\ref{rmk:pro-GL-fibrations} and hence every  elementary matrix in $\GL(B)$ can be lifted to $\GL(A)$. 
By \cite[Exc.~II.2.3]{K-book} we have an exact sequence
\begin{equation}\label{sec:GL-K0}
0 \to \GL(B)/\im(\GL(A)\to\GL(B)) \to K_{0}(I) \to K_{0}(A) \to K_{0}(B).
\end{equation}
Splicing together the long exact sequence of homotopy groups associated to \eqref{seq:GL}, the sequence \eqref{seq:GL-0}, and \eqref{sec:GL-K0}, we get the desired result.
\end{proof}

\begin{cor}
	\label{cor:KV-nilinvariant}
If $I \subset A$ is a closed, nilpotent ideal, then $KV^{\an}_{i}(A) \cong KV^{\an}_{i}(A/I)$ for all $i\geq 1$.
\end{cor}
\begin{proof}
Since $I$ is nilpotent, $K_{0}(I) = 0$. By Remark \ref{rmk:pro-GL-fibrations}(ii) and Proposition \ref{prop:excision}, it suffices to prove that
 $\GL(I\langle \Delta_{\rho}\rangle)$ is contractible for every $\rho$. 
Since $I\langle \Delta^{n}_{\rho}\rangle$ is nilpotent for each $n$ and $\rho$, $\GL(I\langle \Delta_{\rho}\rangle) \cong 1+\mathrm{M}(I\langle \Delta_{\rho}\rangle)$
and the contractibility is shown as in the proof of Claim~\ref{claim.phiwe}.
\end{proof}

Since the maps $A\langle s\rangle_{\rho} \xrightarrow{s\mapsto 0} A$ have a section, they are pro-$\GL$-fibrations. By Propositions \ref{prop:excision} and  \ref{prop:pro-homotopy-invariance} it follows that 
\begin{equation}\label{eq:vanishing-KV-path-space}
\prolim{\rho} KV^{\an}_{i}(sA\langle s\rangle_{\rho}) = 0
\end{equation}
for every $i\geq 1$.

\begin{lemma}\label{lem.proGLfibration}
For $\rho >1$, the map $sA\langle s\rangle_{\rho} \xrightarrow{s\mapsto 1} A$ is a pro-$\GL$-fibration.
\end{lemma}
\begin{proof}
Write $t=(t_{1}, \dots, t_{n})$.
Fix $\rho' \geq 1$. There is a unique continuous homomorphism $A\langle t\rangle_{\rho\rho'} \to A\langle s\rangle_{\rho}\langle t\rangle_{\rho'}$ sending $t_{i}$ to $st_{i}$.
Given a matrix $g=g(t) \in \GL(A\langle t\rangle_{\rho\rho'})$ with $g(0)=1$, we get a matrix $g(st) \in \GL(sA\langle s\rangle_{\rho}\langle t\rangle_{\rho'})$ whose image in $\GL(A\langle t\rangle_{\rho'})$ under the map $s\mapsto 1$ coincides with that of $g$, concluding the proof.
\end{proof}
For $\rho>1$ we set
\[
\Omega_{\rho}A := \ker(sA\langle s\rangle_{\rho}\xrightarrow{s\mapsto 1} A) = s(s-1)A\langle s\rangle_{\rho}.
\]
\begin{cor}
We have natural isomorphisms
\[
KV_{i}^{\an}(A) \cong \prolim{\rho} KV^{\an}_{i-1}(\Omega_{\rho}A) 
\]
for $i\geq 2$, and
\[
KV_{1}^{\an}(A) \cong \prolim{\rho}\ker(K_{0}(\Omega_{\rho}A) \to K_{0}(sA\langle s\rangle_{\rho})).
\]
\end{cor}
\begin{proof}
For every $\rho>1$ we have a long exact excision sequence
\begin{multline*}
\dots \to KV^{\an}_{i}(sA\langle s\rangle_{\rho}) \to KV^{\an}_{i}(A) \to KV^{\an}_{i-1}(\Omega_{\rho}A) \to KV^{\an}_{i-1}(sA\langle s\rangle_{\rho}) \to \\
\dots \to KV^{\an}_{1}(A) \to K_{0}(\Omega_{\rho}A) \to K_{0}(sA\langle s\rangle_{\rho} ) \to K_{0}(A).
\end{multline*}
Now apply \eqref{eq:vanishing-KV-path-space}.
\end{proof}

\bibliographystyle{amsalpha}
\bibliography{Kanalytic}

\end{document}